\documentclass[a4paper,11pt]{article}

\usepackage{stmaryrd}
\usepackage{amsfonts}
\usepackage{bbm}

\usepackage[all]{xy}

\usepackage{amscd}
\usepackage{mathrsfs}
\usepackage{latexsym,amssymb,amsmath,amscd,amscd,amsthm,amsxtra}
\usepackage[dvips]{graphicx}
\usepackage[utf8]{inputenc}
\usepackage[T1]{fontenc}
\usepackage{enumerate}
\usepackage{lmodern}
\usepackage{amssymb}
\usepackage{nicefrac,mathtools,enumitem}
\usepackage{microtype}
\usepackage{amsfonts}
\usepackage{amssymb}
\usepackage{mathrsfs}
\usepackage{amsmath}
\usepackage{amsthm}
\usepackage{enumerate}
\usepackage{indentfirst}
\usepackage{amsfonts}
\usepackage{amssymb}
\usepackage{mathrsfs}
\usepackage{amsmath}
\usepackage{amsthm}
\usepackage{enumerate}
\usepackage{cite}
\usepackage{mathrsfs}
\usepackage{kpfonts}
\usepackage{geometry}
\usepackage{ifpdf}
\ifpdf
  \usepackage[colorlinks=true,linkcolor=blue,citecolor=red, final,backref=page,hyperindex]{hyperref}
\else
  \usepackage[colorlinks,final,backref=page,hyperindex,hypertex]{hyperref}
\fi

\allowdisplaybreaks

\makeatletter
\newcommand{\subjclass}[2][2020]{%
  \let\@oldtitle\@title%
  \gdef\@title{\@oldtitle\footnotetext{#1 \emph{Mathematics subject classification}: #2}}%
}
\newcommand{\keywords}[1]{%
  \let\@@oldtitle\@title%
  \gdef\@title{\@@oldtitle\footnotetext{\emph{Keywords}: #1}}%
}
\makeatother

\allowdisplaybreaks
\newtheorem{defn}{Definition}[section]
\newtheorem{thm}[defn]{Theorem}
\newtheorem{lem}[defn]{Lemma}
\newtheorem{prop}[defn]{Proposition}
\newtheorem{cor}[defn]{Corollary}
\newtheorem{ex}[defn]{Example}
\newtheorem{re}[defn]{Remark}
\newcommand{\bdefn}{\begin{defn}}
\newcommand{\edefn}{\end{defn}}
\newcommand{\bthm}{\begin{thm}}
\newcommand{\ethm}{\end{thm}}
\newcommand{\blem}{\begin{lem}}
\newcommand{\elem}{\end{lem}}
\newcommand{\bprop}{\begin{prop}}
\newcommand{\eprop}{\end{prop}}
\newcommand{\bcor}{\begin{cor}}
\newcommand{\ecor}{\end{cor}}
\newcommand{\beg}{\begin{eg}}
\newcommand{\eeg}{\end{eg}}
\newcommand{\bre}{\begin{re}}
\newcommand{\ere}{\end{re}}
\newcommand{\bpf}{\begin{proof}}
\newcommand{\epf}{\end{proof}}

\newcommand{\benu}{\begin{enumerate}}
\newcommand{\eenu}{\end{enumerate}}
\newcommand{\bc}{\begin{center}}
\newcommand{\ec}{\end{center}}
\newcommand{\bea}{\begin{eqnarray}}
\newcommand{\eea}{\end{eqnarray}}
\newcommand{\ba}{\begin{align*}}
\newcommand{\ea}{\end{align*}}
\newcommand{\Bea}{\begin{eqnarray*}}
\newcommand{\Eea}{\end{eqnarray*}}
\newcommand{\beq}{\begin{equation}}
\newcommand{\eeq}{\end{equation}}
\newcommand{\Beq}{\begin{equation*}}
\newcommand{\Eeq}{\end{equation*}}
\newcommand{\bspl}{\begin{split}}
\newcommand{\espl}{\end{split}}

\numberwithin{equation}{section}

\begin{document}
\date{}
\title{\bf  Yang-Baxter equations and $\mathcal O$-operators of a Hom-Jordan superalgebra with representation}
\author{ S. Mabrouk,  O. Ncib and  S. Sendi}
\author{\normalsize \bf Sami Mabrouk \small{$^{1}$} \footnote{  E-mail: mabrouksami00@yahoo.fr}, Othmen Ncib\small{$^{1}$} \footnote{  E-mail: othmenncib@yahoo.fr (Corresponding author)} and Sihem Sendi \small{$^{2}$} \footnote{  E-mail: sihemsendi995@gmail.com }}
\date{{\small{$^{1}$}   Faculty of Sciences, University of Gafsa,   BP
2100, Gafsa, Tunisia \\  \small{$^{2}$}    Faculty of Sciences, University of Sfax,   BP 1171, 3000 Sfax, Tunisia}}

\maketitle
\begin{abstract}
In this paper, first we recall the notion of  Hom-Jordan superalgebras and study their representations. We define the
 Yang-Baxter equation in a Hom-Jordan superalgebra. Additionally, we  extend the connections between  
 $\mathcal {O}$-operators  and  skew-symmetric solutions
  Yang-Baxter equation  of Hom-Jordan superalgebras (HJYBE). In which,  we prove that a super skew-symmetric solution of  HJYBE
   can be interpreted as an $\mathcal{O}$-operator associated to the coadjoint representation.
 Finally, we study the relationship between Hom-pre-Jordan superalgebras and Hom-Jordan superalgebras via an $\mathcal{O}$-operators. Some other related results are considered.

\end{abstract}


\tableofcontents
\section{Introduction}
Jordan superalgebras were first studied in \cite{VG} by classifying finite-dimensional simple
Jordan superalgebras over an algebraically closed field of characteristic zeros. These algebraic structures have been rapidly developed \cite{VG,RZ}. A Jordan superalgebra is a $\mathbb Z_2$-graded vector space $\mathcal{J}$ equipped with an even super-commutative bilinear map $'\cdot'$ (i.e. $x\cdot y =(-1)^{|x||y|} y \cdot x$) that satisfies the Jordan super-identity:
\begin{align*}
  \sum_{x,y,t}(-1)^{|u|(|x|+|z|)}as(x\cdot y,z,u)=0.
\end{align*}
where $x,y,u\in \mathcal{J}$ and $\displaystyle\sum_{x,y,u}$ denoted the cyclic sum over $(x,y,z)$ and $as(x,y,z)=(x\cdot y)\cdot z-x\cdot(y\cdot z)$ for any $x,y,z\in \mathcal{J}$. The reader is referred to \cite{V} and  \cite{shest-jordan-super} for discussions about the important role of Jordan superalgebras in physics, especially in quantum
mechanics.


Yang-Baxter system plays a crucial role in many fields like
integrable systems, quantum groups, quantum field, and so
on and has become an important topic in both mathematics
and mathematical physics since $1980$s. Classical Yang–Baxter equation (CYBE) has a profound connection with many branches of mathematical physics and pure mathematics (\cite{Belavin-Drinfel,V-Drinfel}). In particular, CYBE can
be regarded as a “classical limit” of quantum Yang–Baxter equation (\cite{Belavin1}), which
plays an important role in the study of classical integrable system (\cite{Gelfand-Dorfman,I-V-Cherednik,M-A-Semenov}). The notion of Jordan $D$-bialgebras was introduced by Zhelyabin \cite{zv1} as an
analogue of Lie bialgebras. A kind of Jordan $D$-bialgebras (coboundary cases)
is obtained from the solutions Jordan Yang-Baxter equation (JYBE), which is an
analogue of the classical Yang-Baxter equation (CYBE) in Lie algebra \cite{zv2}. The original form of a JYBE is given in the tensor form, so it is natural to
consider operator form of the JYBE which satisfies certain conditions \cite{DNC}. It
was proved that a skew-symmetric solution of the JYBE was exactly a special
operator, which was called $\mathcal{O}$-operator. Indeed,
the notion of pre-Jordan superalgebra as a Jordan algebraic analogue of a pre-Lie superalgebra was
introduced in \cite{wang}. 

The notion of Hom-associative algebra were introduced in \cite{ms} is a generalization of associative algebras, it corresponds to the notion of Hom-Lie algebra which appeared in physics, in the sense that the commutator
of a Hom-associative algebra gives a Hom-Lie algebra.
It was originally introduced by Makhlouf and Silvestrov in \cite{ms4}. Hom-Lie algebras and more general quasi-Hom-Lie algebras were introduced first
by Hartwig, Larsson and Silvestrov in \cite{hls} where a general approach to discretization
of Lie algebras of vector fields using general twisted derivations ($\sigma$-derivations) and a
general method for construction of deformations of Witt and Virasoro type algebras
based on twisted derivations have been developed. In \cite{KFA} the authors introduce Hom-alternative, Hom-Malcev and Hom-Jordan superalgebras which are generalizations of Hom-alternative, Hom-Malcev and Hom-Jordan algebras respectively studies in \cite{Makhlouf1, Yau}. The notion of Hom-pre-alternative, Hom-pre-Malcev and Hom-pre-Jordan algebras are introduced in \cite{Chtioui1, Harrathi, Chtioui2}.

The purpose of this paper is to study Hom-type generalizations of  Jordan superalgebras which are endowed with an even nondegenerate supersymmetric associative bilinear form are called pseudo-Euclidean. Super JYBE in Hom-Jordan superalgebras and $\mathcal{O}$-operators are
introduced. Moreover, we exploit Hom-pre-Jordan superalgebras which are the algebraic structures behind the $\mathcal{O}$-operators and study their relations with
Hom-Jordan superalgebras and Hom-dendriform superalgebras.

This paper is organized as follows. In Section \ref{Sec2}, we give some fundamental
results on Hom-Jordan superalgebras and we introduce the notion of Pseudo-Euclidean Hom-Jordan superalgebras. In Section \ref{Sec3}, we introduce
the notion of $\mathcal{O}$-operators of Hom-Jordan superalgebras, and then constuct
a direct relation between $\mathcal{O}$-operator and super-HJYBE with respect to the coadjoint representation. In Section \ref{Sec4},
we introduce The notion of Hom-pre-Jordan superalgebras and give their relations with $\mathcal{O}$-operators, Hom-Jordan superalgebras, and Hom-dendriform superalgebras.\\

Recall that, in general a superalgebra means a $\mathbb{Z}_2$-graded algebra, that is an algebra $\mathcal{A}$ which may be written as a direct sum of subspaces $\mathcal{A}=\mathcal{A}_{0}\oplus\mathcal{A}_{1}$ subject to the relation $\mathcal{A}_{i}\mathcal{A}_{j}\subseteq \mathcal{A}_{i+j}$.
The subspaces $\mathcal{A}_{0}$ and $\mathcal{A}_{1}$ are called the even and the odd parts of the superalgebra $\mathcal{A}$. Define the $\mathbb{Z}_2$-graded vector space $\mathcal{A}^{op}=\mathcal{A}_{1}\oplus\mathcal{A}_{0}$.
Throughout this paper, $\mathbb{K}$ denotes an algebraically closed field of characteristic 0. All algebras and vector spaces are considered on $\mathbb{K}$. The parity of the homogeneous element $x$ is denoted by $|x|$.
 \section{Basics on Hom-Jordan superalgebras}\label{Sec2}
In this section, we presents fundamental concepts on Hom-Jordan superalgebras introduced in \cite{KFA}  and develop some helpful results that we will use later.
\begin{defn}\cite{KFA}
  A Hom-Jordan superalgebra is a triple $(\mathcal{J}, \cdot, \alpha)$
    consisting of $\mathbb{Z}_2$-graded vector space $\mathcal{J}$, an even bilinear map $\cdot: \mathcal{J}\times \mathcal{J}
    \rightarrow \mathcal{J}$ which is super-commutative (i.e  $x\cdot y=(-1)^{|x||y|}y\cdot x$)
satisfying
\begin{equation}\label{Hom-Jordan super-identity}
\sum_{x,y,u}(-1)^{|u|(|x|+|z|)}as_\alpha(x\cdot y,\alpha(z),\alpha(u))=0.
\end{equation}
where $x,y,z,u\in \mathcal{H}(\mathcal{J})$,  $\sum_{x,y,z}$ denoted the cyclic sum over $(x,y,z)$ and $as_\alpha(x,y,z)=(x\cdot y)\cdot \alpha(z)-\alpha(x)\cdot(y\cdot z)$.
A Hom-Jordan superalgebra is called multiplicative if $\alpha$ is a superalgebra morphism (i.e. for any $x,y\in\mathcal{J}$, we have $\alpha(x\cdot y)=\alpha(x)\cdot\alpha(y)$ and called regular if $\alpha$ is a superalgebra automorphism. Throughout this work, all Hom-algebras are considered multiplicative.
\end{defn}

The Hom-jordan super-identity \eqref{Hom-Jordan super-identity} can be written as
\begin{equation}
\sum_{x,y,u}(-1)^{|u|(|x|+|z|)}(((x\cdot y)\cdot\alpha(z))\cdot \alpha^{2}(u))-\alpha(x\cdot y)\cdot(\alpha(z)\cdot\alpha(u)))=0
\end{equation}
A Hom-associative superalgebra is just a Hom-associative algebra with a $\mathbb{Z}_{2}$-graduation but a Hom-jordan superalgebra is not necessarily a Hom-Jordan algebra.
\begin{ex}\label{ex-Jord-sup}
Let $(\mathcal J,\cdot)$ be a Jordan superalgebra and $\alpha:\mathcal J\to \mathcal J$ be an algebra morphism. Then $(\mathcal J,\cdot_\alpha,\alpha)$ is a Hom-Jordan superalgebra, where $x\cdot_\alpha y=\alpha(x)\cdot\alpha (y).$       
\end{ex}
\begin{ex}\cite{KFA}\label{ex-hom-jor}
 We consider the $3$-dimensional Kaplansky superalgebra $K_{3}=\langle\ e\ | \ x,\ y\ \rangle$. The   product is defined as:
$$e\cdot e=e,~e\cdot x=\frac{1}{2}x,~e\cdot y=\frac{1}{2}y,~x\cdot y=e.$$
$K_{3}$ is a simple Jordan superalgebra.\\
Even superalgebra endomorphisms $\alpha$ with respect to the basis $\{e,x,y\}$ are defined by
$$\alpha(e)=e,~~~~\alpha(x)=\frac{1}{c}x,~~~~\alpha(y)=\frac{1}{c}y,$$
with $c\neq 0$.
According to  Example \ref{ex-Jord-sup}, the even linear map $\alpha$ and the following multiplication
$$e\cdot_{\alpha}e=e,~~~e\cdot_{\alpha}x=\frac{1}{2c}x,~~~e\cdot_{\alpha}y=\frac{1}{2c}y,~~~x\cdot_{\alpha}y=\frac{1}{c^2}e.$$
determine a $3$-dimensional Hom-Jordan superalgebra.    
\end{ex}
\begin{prop}\cite{KFA}\label{FromAltToJorMalcev}
Let $(\mathcal{J},\mu, \alpha)$ be a Hom-alternative superalgebra. Then $(\mathcal{J}, \ast , \alpha)$ is a Hom-Jordan superalgebra where
$$x\ast y=\mu(x,y)+(-1)^{|x||y|}\mu(y,x),\quad \forall x,y \in \mathcal{H}(\mathcal{J}).$$ 
\end{prop}

\begin{defn}
Let $(\mathcal{J}, \cdot, \alpha) $ be a Hom-Jordan superalgebra, and $V$ be a $\mathbb{Z}_2$-graded vector space. A $\mathcal{J}$-module on $V$ with respect to $\beta\in End(V)$ is an even linear map  $\pi: \mathcal{J} \to \mathfrak{gl}(V)$   such that, the following conditions are satisfied for any  $ x,y,z\in \mathcal{H}(\mathcal{J})$:
{\small\begin{align}\label{Hom-jordan-representation1}
& \beta  \pi(x)= \pi(\alpha(x))  \beta , \\
\label{Hom-jordan-representation2} &(-1)^{|x||z|}  \pi(\alpha(x)\cdot \alpha(y))\pi(\alpha(z))\beta+(-1)^{|x||y|}\pi(\alpha(y)\cdot \alpha(z))\pi(\alpha(x))\beta + (-1)^{|z||y|} \pi(\alpha(z)\cdot \alpha(x))\pi(\alpha(y))\beta  \\\nonumber
=& (-1)^{|x||z|}\pi(\alpha^2(x))\pi(y\cdot  z)  \beta + (-1)^{|x||y|}\pi(\alpha^2(y))\pi(z\cdot x)  \beta  +(-1)^{|y||z|} \pi(\alpha^2(z))\pi(x\cdot y)  \beta,\\\label{Hom-jordan-representation3}
&  (-1)^{|x||z|} \pi(\alpha(x)\cdot \alpha(y))\pi(\alpha(z)) \beta+(-1)^{|x||y|}\pi(\alpha(y)\cdot \alpha(z))\pi(\alpha(x)) \beta +(-1)^{|z|y||}  \pi(\alpha(z)\cdot \alpha(x))\pi(\alpha(y)) \beta  \\\nonumber
=& (-1)^{|x||z|} \pi(\alpha^2(x))\pi(\alpha(y))\pi(z)+(-1)^{|x||y|+|y||z|}\pi(\alpha^2(z))\pi(\alpha(y))\pi(x)+(-1)^{|z||y|+|x||z|}\pi((x\circ z)\cdot \alpha(y)) \beta^2.
\end{align}}
\end{defn}
\begin{ex}
 Let $(\mathcal{J}, \cdot, \alpha) $ be a Hom-Jordan superalgebra. Then, the even linear map $ad:\mathcal{J}\longrightarrow End(\mathcal{J})$ defined by
 $$ad(x)(y)= x\cdot y,~~\forall x,y\in \mathcal{H}(\mathcal{J}), $$
 is a $\mathcal{J}$-module with respect to $\alpha$ of $(\mathcal{J}, \cdot, \alpha) $ called the adjoint $\mathcal{J}$-module.
 \end{ex}

In fact, $(V,\pi,\beta)$  is a module of a Hom-Jordan superalgebra $(\mathcal{J},\cdot,\alpha)$ if and only if there
exists a Hom-Jordan superalgebra structure on the direct sum $\mathcal{J}\oplus V$
of the underlying vector spaces of $\mathcal{J}$ and $V$
given by
 \begin{align}
     (x+u)\cdot_{\mathcal{J}\oplus V}(y+v)=x\cdot y+\pi(x)(v)+(-1)^{|v||y|}\pi(x)(u),
 \end{align}
 and the twist map
 \begin{align}
\alpha_{\mathcal{J}\oplus V}(x+u)=\alpha(x)+\beta(u)
 \end{align}
for any $x,y\in\mathcal{H}(\mathcal{J})$ and $u,v \in \mathcal{H}(V)$, is a Hom-Jordan superalgebra called semi-direct product of the Hom-Jordan superalgebra $(\mathcal{J},\cdot,\alpha)$ and $V$ which denoted by $\mathcal{J}\ltimes V$.


We recall some facts on vector superspaces from \cite{Bai-Guo-Zhang, Cheng-Wang, M-Scheunert}. Let  $V=V_{\bar 0}\oplus
V_{\bar 1}$ and $W=W_{\bar 0}\oplus W_{\bar 1}$ be two vector superspaces over a field $\mathbb K$.
A bilinear form $\mathfrak B: V\times W \longrightarrow \mathbb K$
is called
 odd  if $\mathfrak B(V_{\bar{0}}, W_{\bar{0}})=\mathfrak B(V_{\bar{1}}, W_{\bar{1}})=0$, and
 even  if $\mathfrak B(V_{\bar{0}}, W_{\bar{1}})=\mathfrak B(V_{\bar{1}}, W_{\bar{0}})=0$;
nondegenerate if  $\mathfrak B(v, w)=0$
for all $w\in W$ implies $v=0$, and $\mathfrak B(v, w)=0$
for all $v\in V$ implies $w=0$.
The linear dual $V^*={\rm Hom}(V, \mathbb  K)$  of $V$ inherits a
$\mathbb{Z}_2$-graduation $V^*=V^*_{\bar
0}\oplus V^*_{\bar 1}$ with
\begin{equation}\label{eq:2.3}
    V^*_\alpha:=\big\{u^*\in V^*| u^*(V_{\alpha+\bar
        1})=\{0\}\big\},\;\; \forall \alpha\in \mathbb{Z}_2.
\end{equation}
Given a linear form $u^*\in V^*$ and a vector $v\in V$, the resulting $u^*(v)$ is also denoted by $\langle u^*,v\rangle$.  It defines an even nondegenerate bilinear form
$\langle -,-\rangle:V^* \times V\longrightarrow \mathbb K$, which is called the canonical pairing.
The canonical pairings induces a pairing  on $V\times V^*$ by
\begin{equation}\label{eq:2.7}\langle v,
    u^*\rangle=(-1)^{|u^*||v|}\langle u^*,v\rangle,\;\;\forall  u^*\in V^*,v\in V.
\end{equation}
We  extend $\langle -,- \rangle$ to $(V \otimes V)^* \times (V \otimes V)$ by setting (applying $(V \otimes V)^* =V^*\otimes V^*$ since $V$ is finite dimensional)
\begin{equation}\label{e.q.1}
\langle u_1^*\otimes u_2^*, v_1\otimes v_2\rangle:=(-1)^{|u_2^*||v_1|}\langle u_1^*,  v_1\rangle\langle u_2^*,  v_2\rangle, \quad \forall u_1^*, u_2^*\in V^*, v_1, v_2 \in V.
\end{equation}

 Let $(V,\pi,\beta)$ be a $\mathcal{J}$-module of $(\mathcal{J}, \cdot, \alpha)$ such that $\beta$ is invertible and $(V^*,\beta^*)$ be the dual of $(V,\beta)$ where
 $\beta^*(u) = u\circ\beta$, for $u\in V^*$. Define
$\pi^*:\mathcal{J}\longrightarrow End(V^*)$ as usual by
\begin{align}
    \langle \pi^*(x)(\xi),u\rangle=(-1)^{|x||\xi|}\langle\xi, \pi(x)u\rangle,~~\forall x\in \mathcal{J},u\in V, \xi\in V^*.
 \end{align}
 Note that in general $\pi^*$ is not a $\mathcal{J}$-module. Define $\pi^{\star}:\mathcal{J}\longrightarrow End(V^*)$ by
\begin{align}\label{def-dual-repr}
\pi^{\star}(x)(\xi)=\pi^*(\alpha(x))((\beta^{-2})^*(\xi)),~~\forall x\in \mathcal{J},\xi\in V^*.
 \end{align}
 More precisely, we have
 \begin{equation}\label{dual-rep-pair}
 \langle \pi^\star(x)(\xi),u\rangle=(-1)^{|x||\xi|}\langle\xi, \pi(\alpha^{-1}(x))(\beta^{-2}u)\rangle,~~\forall x\in \mathcal{J},u\in V, \xi\in V^*.    
 \end{equation}
 \begin{thm}\label{Dual-repres}
Under the above notations  $\pi^\star$   is a $\mathcal{J}$-module on $V^*$ with respect to $(\beta^{-1})^*$, which is called the dual representation of $(V,\pi,\beta)$.

 \end{thm}
    \begin{proof}
 For all $x\in \mathcal{H}(\mathcal{J}),\xi\in \mathcal{H}(\mathcal{J}^*)$,  we have
\begin{eqnarray*}
\pi^\star(\alpha(x))((\beta^{-1})^*(\xi))=\pi^*(\alpha^{2}(x))(\beta^{-3})^*(\xi)=(\beta^{-1})^*(\pi^*(\alpha(x))(\beta^{-2})^*(\xi))
=(\beta^{-1})^*(\pi^\star(x)(\xi)),
\end{eqnarray*}
On the other hand, for all $x,y\in \mathcal{H}(\mathcal{J}), \xi\in \mathcal{H}(\mathcal{J}^*)$ and $v\in V$, we have
\begin{align*}
&\sum_{x,y,z}(-1)^{|x||z|} \pi^{\star}(\alpha(x)\cdot \alpha(y)) \pi^{\star}(\alpha(z))((\beta^{-1})^*(\xi))(v)\\
=&\sum_{x,y,z}(-1)^{|x||z|} \pi^{*}(\alpha^{2}(x)\cdot \alpha^{2}(y)) \pi^{*}(\alpha^{4}(z))((\beta^{-5})^*(\xi))(v))\\
=&(-1)^{(|x|+|y|+|z|)|\xi|}\sum_{x,y,z}(-1)^{|y||z|}\xi(\beta^{-5} ( \pi(\alpha^{4}(z))\pi(\alpha^{2}(x)\cdot \alpha^{2}(y))(v)))\\
=&(-1)^{(|x|+|y|+|z|)|\xi|}\sum_{x,y,z}(-1)^{|y||z|}\xi(\beta^{-3} ( \pi(\alpha^{2}(z))\pi(x\cdot  y)\beta\beta^{-3}(v)))\\
=&(-1)^{(|x|+|y|+|z|)|\xi|}\sum_{x,y,z}(-1)^{|x||y|}\xi(\beta^{-3} ( \pi(\alpha(y)\cdot \alpha(z)) \pi( \alpha(x))\beta\beta^{-3}(v)))\\=&(-1)^{(|x|+|y|+|z|)|\xi|}\sum_{x,y,z}(-1)^{|x||y|}\xi(\beta^{-5} ( \pi(\alpha^3(y)\cdot \alpha^3(z)) \pi( \alpha^3(x))(v)))\\=&\sum_{x,y,z}(-1)^{|x||z|}\pi^{\star}(\alpha^{2}(x))\pi^{\star}(y \cdot z)(\beta^{-1})^{*}(v)
.
\end{align*}
Similarly, we obtain  the other identity.
Therefore, $\pi^{\star}$ is a $\mathcal{J}$-module on  $V^*$ with respect to $(\beta^{-1})^*$.
 \end{proof}
\begin{cor}\label{dual-adj-repr}
   Let $(\mathcal{J},\cdot, \alpha) $ be a Hom-Jordan superalgebra. Then the even linear map $ad^{\star}:\mathcal{J}\longrightarrow End(V^*)$ defined by
 $$ad^{\star}(x)(f)=ad^{*}(\alpha(x))(\alpha^{-2})^{*}(f),~~\forall x\in \mathcal{H}(\mathcal{J}),f\in \mathcal{H}(\mathcal{J}^*)$$
 is a $\mathcal{J}$-module of Hom-jordan superalgebra $(\mathcal{J}, \cdot, \alpha)$ on $\mathcal{J}^*$ with respect to $(\alpha^{-1})^{*}$, which called the coadjoint representation.
 \end{cor}
\section{$\mathcal{O}$-operators and super Hom-Jordan Yang-Baxter equation}\label{Sec3}
In this section, we introduce the notion of $\mathcal O$-operators on Hom-Jordan superalgebras with respect a given representation. Then, we give a characterization of an $\mathcal O$-operator by graphs and Nejinhuis operators on the associated  semi-direct product. We study super Hom-Jordan Yang-Baxter equation (HJYBE) and   establishes the relationship between a super skew-symmetric solution of
super HJYBE  and the corresponding $\mathcal{O}$-operators on
Hom-Jordan superalgebra with respect to the coadjoint representation.
\subsection{Definition and characterizations}

\begin{defn}
Let $(\mathcal{J},\cdot,\alpha)$ be a Hom-Jordan superalgebra and $(V,\pi,\beta)$ be a representation of $\mathcal{J}$.
An even linear map $T: V \to \mathcal{J}$ is called  $\mathcal{O}$-operator on $\mathcal{J}$ associated to $\pi$ if it satisfies
\begin{align}
 & \alpha T  =  T\beta ,\\
\label{OoperatorsCond} & T(u) \cdot T(v)=T \big( \pi(T(u))v + (-1)^{|u||v|}\pi(T(v))u \big),\quad \forall u,v \in \mathcal{H}(V).
\end{align}

\end{defn}
In particular, if $R$ is an $\mathcal{O}$-operator associated to the adjoint representation $(\mathcal{J}, ad, \alpha)$, then $R$ is called a
Rota-Baxter operator
(of weight zero) on $\mathcal{J}$, that is satisfies
\begin{align}
 R(x) \cdot R(y) = R(R(x) \cdot y + x \cdot R(y)),\;\; \forall x, y \in \mathcal{H}(\mathcal{J}).
\end{align}

In the following, we characterize $\mathcal{O}$-operators in terms of their graph.
\begin{prop}
Let   $(V, \pi,\beta)$  be a representation of   Hom-Jordan  superalgebra $\mathcal J$. Then
 an even  linear map $T:V\rightarrow \mathcal J$ is  an $\mathcal{O}$-operator  associated to $(V, \pi,\beta)$ if and only if the graph $$\mathrm{Gr}(T)=\{T(u)+u|~u\in V\}$$ of the map $T$
is a  super subalgebra of the  semi-direct product $\mathcal J\ltimes V$.
\end{prop}
\begin{proof}
 Let $T:V\rightarrow \mathcal J$ be a linear map. Then $\alpha_{\mathcal{J}\oplus V}(\mathrm{Gr}(T))\subset \mathrm{Gr}(T)$ if and only if $\alpha T=T\beta$.
 
 For all $u,v\in \mathcal H(V)$, we have
 \begin{align*}
    (T(u)+u)\cdot_{\pi}(T(v)+v)=&T(u)\cdot T(v)+ \pi(T(u))v+(-1)^{|u||v|}\pi(T(v))u,
 \end{align*}
 which implies that the graph $\mathrm{Gr}(T)=\{T(u)+u|~u\in V\}$ is a super subalgebra of  the Hom-Jordan superalgebra $\mathcal J\ltimes V$ if and only if
 $T$ satisfies
 \begin{align*}
    &[T(u),T(v)]=T\big(\pi(T(u))v+(-1)^{|u||v|}\pi(T(v))u\big),
 \end{align*}
 which means that $T$ is a  $\mathcal{O}$-operator.
\end{proof}

In the sequel, we characterize super  $\mathcal{O}$-operators associated to $(V, \pi,\beta)$ in terms of the Nijenhuis operators.  Recall  that a Nijenhuis operator on a Hom-Jordan superalgebra $\mathcal J$ is an aven linear map $N:\mathcal J\rightarrow \mathcal J$  satisfying, for all $x,y\in \mathcal J$, 
$$N\alpha=\alpha N,$$  $$N(x)\cdot N(y)=N\big(N(x)\cdot y+x\cdot N(y)-N(x\cdot y)\big).$$
\begin{prop}
Let  $(V, \pi,\beta)$  be a representation of Hom-Jordan superalgebra $\mathcal J$. Then
 an even linear map $T:V\rightarrow \mathcal J$ is  an $\mathcal{O}$-operator  associated to $(V, \pi,\beta)$ if and only if
$$N_T=\begin{bmatrix}
    0 & -T \\
    0  & 0
\end{bmatrix}:\mathcal J\oplus V\rightarrow \mathcal J\oplus V $$
is a Nijenhuis operator on the semi-direct product Hom-Jordan superalgebra $\mathcal J\ltimes V$.
\end{prop}

\begin{proof}By direct computation, we can check that $N_T\alpha_{\mathcal{J}\oplus V}=\alpha_{\mathcal{J}\oplus V} N_T$ if and only if $\alpha T=T\beta.$
For all $x,y\in \mathcal J,$ $u,v\in V$, on the one hand, we have
\begin{align*}
& N_T(x+u)\cdot_{\mathcal{J}\oplus V}N_T(y+v)= T(u) \cdot T(v).
\end{align*}
On the other hand, since $N^{2}_{T}=N_{T},$ we have
\begin{align*}
&
N_T\big(N_T(x+u)\cdot_{\mathcal{J}\oplus V} (y+v)+(x+u)\cdot_{\mathcal{J}\oplus V} N_T(y+v )-N_T((x+u)\cdot_{\mathcal{J}\oplus V}(y+v))\big)\\
=& N_T \big((-T(u))\cdot_{\mathcal{J}\oplus V}(y+v)+(x+u)\cdot_{\mathcal{J}\oplus V}(-T(v))-N_T(x\cdot y+\pi(x)v+(-1)^{|u||v|}\pi(y)u)\big)\\
=& T\big(\pi(T(u))v+(-1)^{|u||v|}\pi(T(v))u \big).
\end{align*}
Therefore,  $N_T$ is a Nijenhuis operator on the semi-direct product Hom-Jordan superalgebra $\mathcal J\ltimes V$ if and only if \eqref{OoperatorsCond} is satisfied.
\end{proof}
\subsection{Super Hom-Jordan Yang-Baxter equation and $\mathcal{O}$-operators}

Let $(\mathcal{J},\cdot, \alpha)$ be a regular Hom-Jordan superalgebra, $(V,\pi,\beta)$ be a $\mathcal J$-module and $r=\displaystyle\sum_{i=1}^{n}x_i\otimes y_i  \in \mathcal{J}\otimes \mathcal{J}$ such that $\alpha^{\otimes 2}(r)=r$.
Consider the elements $r_{12} = \displaystyle\sum_{i=1}^{n} x_i \otimes y_i \otimes 1$, $r_{13} =\displaystyle\sum_{i=1}^{n} x_i \otimes 1 \otimes y_i$ and $r_{23} =\displaystyle\sum_{i=1}^{n} 1 \otimes x_i \otimes y_i$ of $\mathcal{J} \otimes \mathcal{J}\otimes \mathcal{J}$,
where 1 is a unit element in $\mathcal{J}$ and 
\begin{align*}
&r_{12}\cdot r_{13}=\sum_{i,j=1}^{n}
(-1)^{|x_j||y_i|}
x_i\cdot x_j \otimes \alpha(y_i) \otimes \alpha(y_j),\\
&r_{12}\cdot r_{23} =\sum_{i,j=1}^{n} \alpha(x_i) \otimes y_i\cdot x_j \otimes \alpha(y_j),\\
&r_{13}\cdot r_{23} = \sum_{i,j=1}^{n}
(-1)^{|x_j||y_i|}
\alpha(x_i) \otimes \alpha(x_j) \otimes y_i\cdot y_j.
\end{align*}
\begin{defn} The standard form of the super Hom-Jordan Yang-Baxter equation (super HJYBE) is given as follows:
\begin{align}\label{form-Hom-Jor-yb-eq}
  r_{12}\cdot r_{13} - r_{12}\cdot r_{23}+ r_{13}\cdot r_{23} = 0.
\end{align}\end{defn}
Let $V$ be a $\mathbb{Z}_{2}$-graded vector space. Let $\sigma:V \otimes V \rightarrow V\otimes V$ be the exchanging operator defined as
\begin{align}
     \sigma (x\otimes y)=(-1)^{|x||y|}y \otimes x,\quad \forall x,y\in V.
\end{align}

If $\sigma(r)=r$ (resp $\sigma(r)=-r$)
for $r\in V\otimes V$, then $r$ is called supersymmetric (resp super skew-symmetric).
We call $r$ solution of the super HJYBE. For a vector superspace V, under
the natural linear isomorphism $V\otimes V\cong Hom(V^*,V)$, any $r\in V\otimes V$ corresponds to a unique linear map $T_r$ from the dual space $V^{*}$ to $V$ such that  
\begin{align}\label{e.q.0}
\langle u^{*}, T_r(v^{*})\rangle=(-1)^{|r||v^{*}|}\langle u^{*}\otimes v^{*}, r\rangle \quad u^{*},v^{*}\in V^{*}
 .
\end{align}

\begin{lem}\label{4.1}
Let $V$ be a $\mathbb{Z}_{2}$-graded vector space. Let $r =\displaystyle\sum_i x_i\otimes y_i\in V\otimes V$ and $v^*\in V^*$. Then $$T_r(v^*)=\displaystyle\sum_i (-1)^{|v^*||y_i|}\langle v^*,y_i\rangle x_i\quad \text{and} \quad \alpha \circ T_r=T_r \circ (\alpha^{-1})^*$$
\end{lem}
  \begin{proof}
Note that  $|r|=|x_i|+|y_i|$. For all $v^*, w^*\in V^*$, we have
\begin{equation*}
\langle w^*,T_r (v^*)\rangle=(-1)^{|r||v^*|}\langle w^*\otimes v^*,
r\rangle= \sum_{i}(-1)^{|y_i||v^*|}\langle
w^*,x_{i}\rangle\langle v^*,y_{i}\rangle=\Big\langle
w^*,\sum_i(-1)^{|y_i||v^*|}\langle v^*,y_i\rangle x_i\Big\rangle.
\end{equation*}
Then $T_r (v^*)=\displaystyle\sum_{i}(-1)^{|y_i||v^*|}\langle v^*,y_i\rangle x_i$. On the other hand,  the condition $\alpha^{\otimes2}r =r $ implies that
 \begin{align*}
 \langle v^*, T_r (w^*)\rangle&=(-1)^{|r ||w^*|}\langle v^*\otimes
w^*,\alpha^{\otimes2}r \rangle=(-1)^{|r ||w^*|}\langle \alpha^*(v^*)\otimes
\alpha^*(w^*),r \rangle\\&=\langle \alpha^*(v^*),
T_r (\alpha^*(w^*))\rangle=\langle v^*, \alpha\circ T_r \circ \alpha^*(w^*)\rangle.
\end{align*} 
Then,  $\alpha^{-1}T_r=T_r\alpha^*$ equivalent to $\alpha T_r=T_r(\alpha^{-1})^*$.
\end{proof}
\begin{prop}\label{sol}
   Let $(\mathcal{J},\cdot, \alpha )$ be a Hom-Jordan superalgebra. Then, an even element $r\in V\otimes V$ is a super skew-symmetric solution of super HJYBE  if and only if the associated operator $T_r:\mathcal{J}^*\rightarrow \mathcal{J}$, defined by equation \eqref{e.q.0}, is an $\mathcal{O}$-operator on $(\mathcal{J}, \cdot, \alpha)$ with
respect to the coadjoint representation
$(\mathcal{J}^*,ad^\star , (\alpha^{-1})^*)$.
\end{prop}
\begin{proof}
Let $\xi,\eta,\gamma\in \mathcal{H}(\mathcal{J}^*)$, and $r=\displaystyle\sum_{i}x_{i}\otimes y_{i}$, by Lemma  \ref{4.1}, we have  $\alpha T_r=T_r(\alpha^{-1})^*$,
 and 
\begin{eqnarray*}
\langle \xi\otimes \eta\otimes \gamma, r_{12}\cdot r_{13}\rangle &=&\sum_{i,j}(-1)^{|x_j||y_i|}\langle \xi\otimes \eta\otimes \gamma, (x_i\cdot x_j)\otimes \alpha(y_i)\otimes \alpha(y_j)\rangle\\
&\stackrel{\eqref{e.q.1}}{=}& \sum_{i,j}(-1)^{|x_j||y_i|+(|x_i|+|x_j|)(|\eta|+|\gamma|)+|y_i||\gamma|}\langle \xi,x_i\cdot x_j\rangle \langle \eta, \alpha(y_i)\rangle \langle \gamma, \alpha(y_j)\rangle\\
&=&\sum_{i,j}(-1)^{|\eta||r|} \big\langle \xi,((-1)^{|\gamma||y_j|}\langle \gamma,\alpha(y_j)\rangle x_j)\cdot((-1)^{|\eta||y_i|}\langle \eta, \alpha(y_i)\rangle x_i)\big\rangle \\
&=&\sum_{i,j} \langle \xi,((-1)^{|\gamma||y_j|}\langle \alpha^*(\gamma),y_j\rangle x_j)\cdot((-1)^{|\eta||y_i|}\langle \alpha^*(\eta), y_i\rangle x_i)\rangle \\
&=&(-1)^{|\eta||\gamma|}\langle \xi,T_r(\alpha^*(\gamma))\cdot T_r(\alpha^*(\eta))\rangle,
\end{eqnarray*}
Similarly,  we have
\begin{eqnarray*}
&&\langle \xi\otimes \eta\otimes \gamma, -r_{12}\cdot r_{23}\rangle = -(-1)^{|\xi||\eta|} \langle \eta,T_r(\alpha^*(\xi))\cdot T_r(\alpha^*(\gamma))\rangle,\\
and
&&\langle \xi\otimes \eta\otimes \gamma, r_{13}\cdot r_{23}\rangle = (-1)^{|\xi||\eta|+|\xi||\gamma|+|\eta||\gamma|} \langle \gamma,T_r(\alpha^*(\eta))\cdot T_r(\alpha^*(\xi))\rangle.
\end{eqnarray*}
Hence
\begin{align}
\langle \xi\otimes \eta\otimes \gamma, r_{12}r_{13}-r_{12}r_{23}+r_{13}r_{23}\rangle=&(-1)^{|\eta||\gamma|}\Big(\langle \xi,T_r(\alpha^*(\gamma))\cdot T_r(\alpha^*(\eta))\rangle\label{eq:1}\\&-(-1)^{|\eta|(|\gamma|+|\xi|)} \langle \eta,T_r(\alpha^*(\xi))\cdot T_r(\alpha^*(\gamma))\rangle\nonumber\\
&+(-1)^{|\xi|(|\eta|+|\gamma|)} \langle \gamma,T_r(\alpha^*(\eta))\cdot T_r(\alpha^*(\xi))\rangle\Big).\nonumber
\end{align}
On the other hand, by using Lemma \ref{4.1}, we have

\begin{eqnarray*}
\langle \eta, T_r(ad^\star(T_r(\alpha^*(\xi)))\alpha^*(\gamma))\rangle&=&\langle \eta, T_r(ad^\star(\alpha^{-1}T_r(\xi))\alpha^*(\gamma))\rangle\\&=&\langle \eta, T_r(\alpha^* ad^\star(T_r(\xi))\gamma)\rangle\\&=&\langle \eta, T_r(\alpha^* ad^*(\alpha T_r(\xi))(\alpha^{-2})^*(\gamma))\rangle\\&=&\langle (\alpha^{-1})^*(\eta), T_r( ad^*( T_r((\alpha^{-1})^*(\xi)))(\alpha^{-2})^*(\gamma))\rangle\\&=&\langle T_r((\alpha^{-1})^*(\eta)),  ad^*( T_r((\alpha^{-1})^*(\xi)))(\alpha^{-2})^*(\gamma)\rangle\\&=&(-1)^{|\eta||\xi|}\langle ad( T_r((\alpha^{-1})^*(\xi)))T_r((\alpha^{-1})^*(\eta)),  (\alpha^{-2})^*(\gamma)\rangle\\&=&(-1)^{|\eta||\xi|}\langle\alpha^{-2} ad( T_r((\alpha^{-1})^*(\xi)))T_r((\alpha^{-1})^*(\eta)),  \gamma\rangle\\&=&(-1)^{|\eta||\xi|}\langle ad( T_r(\alpha^*(\xi)))T_r(\alpha^*(\eta)),  \gamma\rangle\\&=&(-1)^{|\eta|(|\gamma|+|\xi|)}\langle  \xi,T_r(\alpha^*(\gamma))\cdot T_r(\alpha^*(\eta))  \rangle
\end{eqnarray*}
 and 
\begin{eqnarray*}
\langle \eta, T_r(ad^\star(T_r(\alpha^*(\gamma))\alpha^*(\xi))\rangle
=(-1)^{|\eta||\gamma|}\langle  \gamma,T_r(\alpha^*(\eta))\cdot T_r(\alpha^*(\xi))  \rangle.
\end{eqnarray*}
Then
\begin{eqnarray}\label{eq:2}
&&\langle \eta, T_r(\alpha^*(\xi))\cdot T_r(\alpha^*(\gamma))-T_r(ad^\star(T_r(\alpha^*(\xi))\alpha^*(\gamma))-(-1)^{|\gamma||\xi|}T_r(ad^\star(T_r(\alpha^*(\gamma))\alpha^*(\xi))\rangle\\\nonumber
&=& \langle \eta, T_r(\alpha^*(\xi))\cdot T_r(\alpha^*(\gamma))-\langle \eta,T_r(ad^\star(T_r(\alpha^*(\xi))\alpha^*(\gamma))-(-1)^{|\gamma||\xi|}\langle \eta,T_r(ad^\star(T_r(\alpha^*(\gamma))\alpha^*(\xi))\rangle
\\&=&-(-1)^{|\eta||\xi|}\Big((-1)^{|\eta||\gamma|}\langle \xi,T_r(\alpha^*(\gamma))\cdot T_r(\alpha^*(\eta))\rangle-(-1)^{|\eta||\xi|} \langle \eta,T_r(\alpha^*(\xi))\cdot T_r(\alpha^*(\gamma))\rangle\nonumber\\\nonumber&&+(-1)^{|\xi|(|\eta|+|\gamma|)+|\eta||\gamma|} \langle \gamma,T_r(\alpha^*(\eta))\cdot T_r(\alpha^*(\xi))\rangle\Big)\\&=&-(-1)^{|\eta||\xi|}\Big(\langle \xi\otimes \eta\otimes \gamma, r_{12}r_{13}-r_{12}r_{23}+r_{13}r_{23}\rangle\Big).
\end{eqnarray}

Set $v^*=\alpha^*(\gamma),w^*=\alpha^*(\eta)$, then follows from Eqs. ~(\ref{eq:1})  and ~(\ref{eq:2}, $r$  is a super skew-symmetric solution of the HJYBE  if and only if  $T_r$  is an $\mathcal{O}$-operator on $(\mathcal{J}, \cdot, \alpha)$ with
respect to the coadjoint representation
$(\mathcal{J}^*,ad^\star , (\alpha^{-1})^*)$.
\end{proof}
Let $\{u_i,u_j\}$ be a basis of $V$, where $u_i\in V_{0}$
 and $v_j\in V_{1}$ and  $\{u^{*}_i,v^{*}_j,\}$ its associated  dual basis with
$u^{*}_i\in V^*_{0}$ and $v^*_{j}\in V^*_{1}$. We set $T:V\rightarrow \mathcal{J}$
such that $T\beta=\alpha T$. By the fact that, as a $\mathbb{Z}_{2}$-graded vector space $Hom(V,\mathcal{J})\cong \mathcal{J}\otimes V^*$, we have
{\small\begin{align}\label{yang}
    T=\sum^{p}_{i=1}T(u_i)\otimes u^{*}_{i}+\sum^{q}_{s=1}T(v_s)\otimes v^{*}_{s}\in\mathcal{J}\otimes V^* \subset (\mathcal{J}\ltimes_{\pi} V^*)^{\otimes2}.
\end{align}}

\begin{thm}\label{oper}
With the above assumption, we have $r=T-\sigma(T)$ is a super skew-symmetric solution of super HJYBE in the Hom-Jordan superalgebra $\mathcal{J}\ltimes_{\pi^\star} V^*$ if and only if, $T$ is an $\mathcal{O}$-operator associated to $\pi^\star$ with respect $(\beta^{-1})^*$.
\end{thm}
\begin{proof}
First, according to \eqref{yang},  we have
$$r=T-\sigma(T)=\sum^{p}_{i=1}T(u_i)\otimes u^{*}_{i}+\sum^{q}_{s=1}T(v_s)\otimes v^{*}_{s}-\sum^{p}_{i=1}u^{*}_{i}\otimes T(u_i)+\sum^{q}_{s=1}v^{*}_s\otimes  T(v_{s}).$$
Then
\begin{align*}
r_{12}\cdot r_{13}=&\sum_{i,j=1}^{p}\{T(u_i)\cdot T(u_j)\otimes (\beta^{-1})^{*}(u^{*}_{i})\otimes (\beta^{-1})^{*}(u^{*}_j)-\pi^\star(T(u_i))u^{*}_j
\otimes (\beta^{-1})^{*}(u^{*}_i)\otimes \alpha(T(u_j))\\&-\pi^\star(T(u_j ))u^{*}_i\otimes \alpha(T(u_i))\otimes (\beta^{-1})^{*}(u^{*}_j)\}+\sum^{p}_{i=1}\sum^{q}_{t=1}\{T(u_i)\cdot T(v_t)\otimes (\beta^{-1})^{*}(u^{*}_{i})\otimes (\beta^{-1})^{*}(v^{*}_t)\\&+\pi^\star(T(u_i)v^{*}_t
\otimes (\beta^{-1})^{*}(u^{*}_i)\otimes \alpha(T(v_t))-\pi^\star(T(v_t ))u^{*}_i\otimes \alpha(T(u_i))\otimes (\beta^{-1})^{*}(v^{*}_t)\}\\&+\sum^{q}_{s=1}\sum^{p}_{j=1}\{T(v_s)T(u_j)\otimes (\beta^{-1})^{*}(v^{*}_{s})\otimes (\beta^{-1})^{*}(u^{*}_j)-\pi^\star(T(v_s)u^{*}_j
\otimes (\beta^{-1})^{*}(v^{*}_s)\otimes \alpha(T(u_j))\\&+\pi^\star(T(u_j ))v^{*}_s\otimes \alpha(T(v_s))\otimes (\beta^{-1})^{*}(u^{*}_j)\}+\sum_{s,t=1}^{q}\{-T(v_s)\cdot T(v_t)\otimes (\beta^{-1})^{*}(v^{*}_{s})\otimes (\beta^{-1})^{*}(v^{*}_t)\\&-\pi^\star(T(v_s)v^{*}_t
\otimes (\beta^{-1})^{*}(v^{*}_s)\otimes \alpha(T(v_t))+\pi^\star(T(v_t ))v^{*}_s\otimes \alpha(T(v_s))\otimes (\beta^{-1})^{*}(v^{*}_t)\}
\end{align*}
By  definition of $\pi^\star$, we have
$$\pi^\star(T(u_i))u_j^*=\pi^{*}(\alpha(T(u_i)))(\beta^{-2})^*(u^{*}_{j})=\sum^{p}_{k=1}(\beta^{-2})^*(u^{*}_{j})(\pi(T(\beta(u_i))))u_{k})u^{*}_{k}.$$

So,
\begin{align*}\sum_{i,j=1}^{p}\pi^\star(T(u_i))u^{*}_j
\otimes (\beta^{-1})^{*}(u^{*}_i)\otimes \alpha(T(u_j))&=
  \sum^{p}_{i,j=1}\pi^{*}(\alpha(T(u_i)))(\beta^{-2})^*(u^{*}_{j})\otimes (\beta^{-1})^{*}(u^{*}_{i})\otimes \alpha(T(u_j))\\&=\sum^{p}_{i,j,k=1}(\beta^{-2})^*(u^{*}_{j})(\pi(T(\beta(u_i)))u_{k})u^{*}_{k}\otimes (\beta^{-1})^{*}(u^{*}_{i})\otimes \alpha(T(u_j))\\&=\sum^{p}_{i,j,k=1}u^{*}_{j}(\pi(T(\beta^{-1}(u_i)))\beta^{-2}(u_{k}))u^{*}_{k}\otimes (\beta^{-1})^{*}(u^{*}_{i})\otimes \alpha(T(u_j))\\&=\sum^{p}_{i,k=1}u^{*}_{k}\otimes (\beta^{-1})^{*}(u^{*}_{i})\otimes \alpha T(\sum^{p}_{j=1}u^{*}_{j}(\pi(T(\beta^{-1}(u_i)))\beta^{-2}(u_{k}))u_j))\\&=\sum^{p}_{i,k=1}u^{*}_{k}\otimes (\beta^{-1})^{*}(u^{*}_{i})\otimes \alpha T(\pi(T(\beta^{-1}(u_i))))\beta^{-2}(u_{k})\\&=\sum^{p}_{i,k=1}u^{*}_{k}\otimes (\beta^{-1})^{*}(u^{*}_{i})\otimes \alpha T(\pi(T(u_i)))\beta^{-1}(u_{k})\\&=\sum^{p}_{i,k=1}u^{*}_{k}\otimes \sum^{p}_{j=1}(\beta^{-1})^{*}(u^{*}_{i})u_ju_j^*\otimes \alpha T(\pi(T(u_i)))\beta^{-1}(u_{k})\\&=\sum^{p}_{j,k=1}u^{*}_{k}\otimes u_j^*\otimes \alpha T(\pi(T(\sum^{p}_{i=1}(\beta^{-1})^{*}(u^{*}_{i})(u_j)u_i)))\beta^{-1}(u_{k})\\&=\sum^{p}_{j,k=1}u^{*}_{k}\otimes u_j^*\otimes T(\pi(T(\beta^{-1}(u_j))))\beta^{-1}(u_{k}).
\end{align*}
Where in the last step is from the fact that pairs $(\beta(u_k),u^*_{k})$ and $(u_k,
(\beta^{-1})^{*}(u^{*}_k))$ are equivalent to each other in the sense of dual basis.
Then we get
\begin{align*}
  r_{12}\cdot r_{13}=&\sum_{i,j=1}^{p}\{T(\beta^{-1}(u_i))\cdot T(\beta^{-1}(u_j))\otimes u^{*}_{i}\otimes u^{*}_j-u^{*}_{i}\otimes u^{*}_{j}\otimes T(\pi(T(\beta^{-1}(u_j))))\beta^{-1}(u_{i})\\&- u^{*}_i\otimes T(\pi(T(\beta^{-1}(u_j)))\beta^{-1}(u_i))\otimes u^{*}_j\}+\sum^{p}_{i=1}\sum^{q}_{t=1}\{T(\beta^{-1}(u_i))\cdot T(\beta^{-1}(v_t))\otimes u^{*}_{i}\otimes v^{*}_t\\&+u^{*}_{i}
\otimes v^{*}_t\otimes T(\pi(T(\beta^{-1}(v_t)))\beta^{-1}(u_i))-u^{*}_i\otimes T(\pi(T(\beta^{-1}(v_t)))\beta^{-1}(u_i))\otimes v^{*}_t\}\\&+\sum^{q}_{s=1}\sum^{p}_{j=1}\{T(\beta^{-1}(v_s))\cdot T(\beta^{-1}(u_j))\otimes v^{*}_{s}\otimes u^{*}_j+v^{*}_s\otimes u^{*}_j\otimes T(\pi(T(\beta^{-1}(u_j)))\beta^{-1}(u_s))\\&+ v^{*}_{s}\otimes T(\pi T(\beta^{-1}(u_j)))\beta^{-1}(u_s))\otimes u^{*}_j\}+\sum_{s,t=1}^{q}\{-T(\beta^{-1}(v_s))\cdot T(\beta^{-1}(v_t))\otimes v^{*}_{s}\otimes v^{*}_t\\&-v^{*}_s\otimes v^{*}_{t} \otimes T(\pi(T(\beta^{-1}(v_t)))\beta^{-1}(u_s)))- v^{*}_s\otimes T(\pi(T(\beta^{-1}(v_t)))\beta^{-1}(u_s))\otimes v^{*}_t\}.
\end{align*}
By the same way, we have:
\begin{align*}
r_{13}\cdot r_{23}=&\sum_{i,j=1}^{p}\{-T(\pi(T(\beta^{-1}(u_i)))\beta^{-1}(u_j))\otimes u^{*}_{i}\otimes u^{*}_{j}+ u^{*}_{i}\otimes u^{*}_j\otimes T(\beta^{-1}(u_i))\cdot T(\beta^{-1}(u_j))\}\\&- u^{*}_i\otimes T(\pi(T(\beta^{-1}(u_i)))\beta^{-1}(u_j))\otimes u^{*}_j\}+\sum^{p}_{i=1}\sum^{q}_{t=1}\{-T(\pi(T(\beta^{-1}(u_i)))\beta^{-1}(u_t))\otimes u^{*}_{i}\otimes v^{*}_{t}\\&+ u^{*}_{i}\otimes v^{*}_t\otimes T(\beta^{-1}(u_i))\cdot T(\beta^{-1}(v_t))- u^{*}_i\otimes T(\pi(T(\beta^{-1}(u_i)))\beta^{-1}(u_t))\otimes v^{*}_t\}\\&+\sum^{q}_{s=1}\sum^{p}_{j=1}\{-T(\pi(T(\beta^{-1}(v_s)))\beta^{-1}(u_j))\otimes v^{*}_{s}\otimes u^{*}_{j}+ v^{*}_{s}\otimes u^{*}_j\otimes T(\beta^{-1}(v_s))\cdot T(\beta^{-1}(u_j))\\&+v^{*}_s\otimes T(\pi(T(\beta^{-1}(v_s)))\beta^{-1}(u_j))\otimes u^{*}_j\}+\sum^{q}_{s,t=1}\{T(\pi(T(\beta^{-1}(v_s)))\beta^{-1}(u_t))\otimes v^{*}_{s}\otimes v^{*}_{t}\\&-v^{*}_{s}\otimes v^{*}_t\otimes T(\beta^{-1}(v_s))\cdot T(\beta^{-1}(v_t))+ v^{*}_s\otimes T(\pi(T(\beta^{-1}(v_s)))\beta^{-1}(u_t))\otimes v^{*}_t\},
\end{align*}
and
\begin{align*}
-r_{12}\cdot r_{23}=&\sum_{i,j=1}^{p}\{-T(\pi(T(\beta^{-1}(u_j)))\beta^{-1}(u_i))\otimes u^{*}_{i}\otimes u^{*}_{j}- u^{*}_{i}\otimes u^{*}_j\otimes T(\pi(T(\beta^{-1}(u_i)))\beta^{-1}(u_j))\\&+ u^{*}_i\otimes T(\beta^{-1}(u_i))\cdot T(\beta^{-1}(u_j))\otimes u^{*}_j\}+\sum^{p}_{i=1}\sum^{q}_{t=1}\{-T(\pi(T(\beta^{-1}(v_t)))\beta^{-1}(u_i))\otimes u^{*}_{i}\otimes v^{*}_{t}\\&+ u^{*}_{i}\otimes v^{*}_t\otimes T(\pi(T(\beta^{-1}(u_i)))\beta^{-1}(u_t))+ u^{*}_i\otimes T(\beta^{-1}(u_i))\cdot T(\beta^{-1}(v_t))\otimes v^{*}_t\}\\&+\sum^{q}_{s=1}\sum^{p}_{j=1}\{-T(\pi(T(\beta^{-1}(u_j)))\beta^{-1}(u_s))\otimes v^{*}_{s}\otimes u^{*}_{j}+ v^{*}_{s}\otimes u^{*}_j\otimes T(\pi(T(\beta^{-1}(v_s)))\beta^{-1}(u_j))\\&- v^{*}_s\otimes T(\beta^{-1}(v_s))\cdot T(\beta^{-1}(u_j))\otimes u^{*}_j\}+\sum^{q}_{s,t=1}\{-T(\pi(T(\beta^{-1}(v_t)))\beta^{-1}(u_s))\otimes v^{*}_{s}\otimes v^{*}_{t}\\&+v^{*}_{s}\otimes v^{*}_t\otimes T(\pi(T(\beta^{-1}(\beta^{-1}(u_s))))v_t)-  v^{*}_s\otimes T(\beta^{-1}(v_s))\cdot T(\beta^{-1}(v_t))\otimes v^{*}_t\},
\end{align*}
Therefore,
\begin{align*}
&r_{12}r_{13}+r_{13}r_{23}-r_{12}r_{23}\\=&\sum_{i,j=1}^{p}\{\big(T(\beta^{-1}(u_i))\cdot T(\beta^{-1}(u_j))-T(\pi(T(\beta^{-1}(u_i)))\beta^{-1}(u_j))-T(\pi(T(\beta^{-1}(u_j)))\beta^{-1}(u_i)\big)\otimes u^{*}_{i}\otimes u^{*}_{j}\\&+u^{*}_i\otimes u^{*}_{j}\otimes \big(T(\beta^{-1}(u_i))\cdot T(\beta^{-1}(u_j))-T(\pi(T(\beta^{-1}(u_i)))\beta^{-1}(u_j))-T(\pi(T(\beta^{-1}(u_j)))\beta^{-1}(u_i)\big)\\&+u^{*}_{i}\otimes \big(T(\beta^{-1}(u_i))\cdot T(\beta^{-1}(u_j))-T(\pi(T(\beta^{-1}(u_i)))\beta^{-1}(u_j))-T(\pi(T(\beta^{-1}(u_j)))\beta^{-1}(u_i)\big)\otimes u^{*}_{j}\}\\&+\sum^{p}_{i=1}\sum^{q}_{t=1}\{\big(T(\beta^{-1}(u_i))\cdot T(\beta^{-1}(v_t))-T(\pi(T(\beta^{-1}(u_i)))\beta^{-1}(u_t))-T(\pi(T(\beta^{-1}(v_t)))\beta^{-1}(u_i)\big)\otimes u^{*}_{i}\otimes v^{*}_{t}\\&-u^{*}_i\otimes v^{*}_{t}\otimes \big(T(\beta^{-1}(u_i))\cdot T(\beta^{-1}(v_t))-T(\pi(T(\beta^{-1}(u_i)))\beta^{-1}(u_t))-T(\pi(T(\beta^{-1}(v_t))))\beta^{-1}(u_i)\big)\\&+v^{*}_{t}\otimes \big(T(\beta^{-1}(u_i))\cdot T(\beta^{-1}(v_t))-T(\pi(T(\beta^{-1}(u_i)))\beta^{-1}(u_t))-T(\pi(T(\beta^{-1}(v_t)))\beta^{-1}(u_i)\big)\otimes v^{*}_{t}\}\\&+\sum^{q}_{s=1}\sum^{p}_{j=1}\{\big(T(\beta^{-1}(v_s))\cdot T(\beta^{-1}(u_j))-T(\pi(T(\beta^{-1}(v_s)))\beta^{-1}(u_j))-T(\pi(T(\beta^{-1}(u_j)))\beta^{-1}(u_s)\big)\otimes v^{*}_{s}\otimes u^{*}_{j}\\&-v^{*}_s\otimes u^{*}_{j}\otimes \big(T(\beta^{-1}(v_s))\cdot T(\beta^{-1}(u_j))-T(\pi(T(\beta^{-1}(v_s)))\beta^{-1}(u_j))-T(\pi(T(\beta^{-1}(u_j)))\beta^{-1}(u_s)\big)\\&-v^{*}_{s}\otimes \big(T(\beta^{-1}(v_s))\cdot T(\beta^{-1}(u_j))-T(\pi(T(\beta^{-1}(v_s)))\beta^{-1}(u_j))-T(\pi(T(\beta^{-1}(u_j)))\beta^{-1}(u_s)\big)\otimes u^{*}_{j}\}\\&+\sum^{q}_{s,t=1}\{-\big(T(\beta^{-1}(v_s))\cdot T(\beta^{-1}(v_t))-T(\pi(T(\beta^{-1}(v_s)))\beta^{-1}(u_t))+T(\pi(T(\beta^{-1}(v_t)))\beta^{-1}(u_s)\big)\otimes v^{*}_{s}\otimes v^{*}_{t}\\&-v^{*}_s\otimes v^{*}_{t}\otimes \big(T(\beta^{-1}(v_s))\cdot T(\beta^{-1}(v_t))-T(\pi(T(\beta^{-1}(v_s)))\beta^{-1}(u_t))+T(\pi(T(\beta^{-1}(v_t)))\beta^{-1}(u_s)\big)\\&-v^{*}_{s}\otimes \big(T(\beta^{-1}(v_s))\cdot T(\beta^{-1}(v_t))-T(\pi(T(\beta^{-1}(v_s)))\beta^{-1}(u_t))-T(\pi(T(\beta^{-1}(v_t)))\beta^{-1}(u_s)\big)\otimes v^{*}_{t}\}.
\end{align*}
So, by the fact that $\beta$ is bijective then, $r$ is a super
skew-symmetric solution of super HJYBE in the Hom-Jordan superalgebra $\mathcal{J}\ltimes_{\pi^\star} V^*$ if and only if T is an $\mathcal{O}$-operator associated to $\pi$.
\end{proof}
A direct conclusion from Proposition \ref{sol} and Theorem \ref{oper} is given as follows.
\begin{cor}
Let $(\mathcal{J},\cdot,\alpha)$ be a Hom-Jordan superalgebra and $(V,\pi,\beta)$ be representation. Let $T: V \to \mathcal{J}$ be an even linear map. Then, the following three conditions are equivalent:
\begin{itemize}
    \item[(i)] T is an $\mathcal{O}$-operator of $\mathcal{J}$ associated to $\pi$.
    \item[(ii)] $T-\sigma(T)$ is a super skew-symmetric solution of super HJYBE in the Hom-Jordan superalgebra $\mathcal{J}\ltimes_{\pi^\star}V^*$.
    \item[(iii)] $T-\sigma(T)$ is an $\mathcal{O}$-operator of the Hom-Jordan superalgebra $\mathcal{J}\ltimes_{\pi^\star}V^{*}$ associated to $ad^{\star}$.
\end{itemize}
\end{cor}

If $r$ is an even super skew-symmetric solution of super HJYBE, then from Proposition \ref{sol}, $T_r:\mathcal{J}^*\rightarrow \mathcal{J}$ is an $\mathcal{O}$-operator on the Hom-Jordan superalgebra $(\mathcal{J}, \cdot, \alpha)$ with
respect to the coadjoint representation
$(\mathcal{J}^*,ad^\star , (\alpha^{-1})^*)$. Thus, from Proposition \ref{induced hompre}, the triplet
$(\mathcal{J}^*,\circ, (\alpha^{-1})^*)$is a Hom-pre-Jordan superalgebra, where $\xi \circ \eta=ad^\star(T_r(\xi))\eta$.  Then there exists an associated Hom-Jordan superalgebra structure on $\mathcal{J}^*$ given by
\begin{align}
   \xi \cdot \eta=ad^\star(T_r(\xi))\eta+(-1)^{|\xi||\eta|}ad^\star(T_r(\eta))\xi \quad \forall \xi,\eta\in \mathcal{H}(\mathcal{J}^*).
  \end{align}
  In the following, we give the notion of symplectic Hom-Jordan superalgebras and we provide some constructions.
\begin{defn}
Let $(\mathcal{J},\cdot,\alpha)$ be a Hom-Jordan superalgebra. An even bilinear form $\mathfrak{B}:\mathcal{J}\times\mathcal{J}\to \mathbb K$ is called
\begin{enumerate}
\item[(1)] supersymmetric if $\mathfrak{B}(x,y)=(-1)^{|x||y|}\mathfrak{B}(y,x),\,\,\ \forall x,y\in\mathcal{H}(\mathcal{J})$
\item[(2)] super-skewsymmetric if $\mathfrak{B}(x,y)=-(-1)^{|x||y|}\mathfrak{B}(y,x),\;\forall x,y\in\mathcal{H}(\mathcal{J})$.
\end{enumerate}
If $\mathfrak{B}$ is  nondegenerate supersymmetric, bilinear form on $\mathcal{J}$ satisfying
\begin{align}\label{alpha-symemetri}\mathfrak{B}(\alpha(x),\alpha(y))&=\mathfrak{B}(x, y),\\\label{2-cocycle}
\mathfrak{B}(x\cdot y,\alpha(z))&=\mathfrak{B}(\alpha(x),y\cdot z)+(-1)^{|x|(|y|+|z|)}\mathfrak{B}(\alpha(y),z\cdot x)\quad \forall x,y,z \in \mathcal{H}(\mathcal{J}).
  \end{align}

 We say that $(\mathcal{J},\cdot, \alpha ,\mathfrak{B})$ is a symplectic Hom-Jordan superalgebra.
\end{defn}
\begin{lem}\label{NondegForm}
  Let $(\mathcal{J},\cdot, \alpha )$ be a Hom-Jordan superalgebra. Any invertible linear map  $T:\mathcal{J}^{*}\rightarrow \mathcal{J}$ induces a nondegenerate bilinear form $\mathfrak{B}$ on $\mathcal{J}$ by:
  \begin{align}\label{quadratic}
  \mathfrak{B}(x,y)=\langle T^{-1}(x),y\rangle,\quad \forall x,y\in  \mathcal{J}.
  \end{align}
  Furthermore, $T$ is called supersymmetric (resp. super skew-symmetric) if the induced bilinear form $\mathfrak{B}$ is supersymmetric (resp. super skew-symmetric).
\end{lem}
\begin{prop}
  Let $(\mathcal{J},\cdot,\alpha)$ be a Hom-Jordan superalgebra and $r\in\mathcal{J}\otimes \mathcal{J}$ be a super skew-symmetric nondegenerate. Then $r$ is a solution of super HJYBE satisfying $\alpha^{\otimes2}r=r$ in $\mathcal{J}$ if and only if, there exists a symplectic  bilinear form $\mathfrak{B}$ on $\mathcal{J}$ defined by $\mathfrak{B}(x,y)=\langle T_r^{-1}(x),y\rangle$.
  \end{prop}
\begin{proof}
By Lemma \ref{NondegForm}, $\mathfrak{B}$ is a nondegenerate super skew-symmetric bilinear since $r$ is nondegenerate super skew-symmetric.  By Proposition \ref{sol}, $T_r$ is an $\mathcal O$-operator, then we have $$\mathfrak{B}(\alpha(x),  \alpha(y))=\langle T_r^{-1}(\alpha(x)),\alpha(y)\rangle=\langle (\alpha^{-1})^*(T_r^{-1}(x)),\alpha(y)\rangle=\langle T_r^{-1}(x),y\rangle=\mathfrak{B}(x,y).$$
On the other hand, for any $x,y,z\in \mathcal{H}(\mathcal{J})$, there exist $\eta_1,\eta_2,\eta_3\in \mathcal{H}(\mathcal{J^{*}})$ such that $x=T_r(\eta_1),y=T_r(\eta_2),z=T_r(\eta_3)$, then  we have
\begin{align*}
    \mathfrak{B}(x\cdot y, \alpha(z))&=\langle T_r^{-1}(x\cdot y), \alpha(z)\rangle\\&=\langle T_r^{-1}(T_r(\eta_1)\cdot T_r(\eta_2)), \alpha(z)\rangle\\&=\langle T_r^{-1}(T_r(ad^\star(T_r(\eta_1))\eta_2+(-1)^{|\eta_1||\eta_2|}ad^\star(T_r(\eta_2)\eta_1)), \alpha(z)\rangle\\&=\langle ad^\star(T_r(\eta_1))\eta_2,\alpha(z)\rangle+(-1)^{|\eta_1||\eta_2|}\langle ad^\star(T_r(\eta_2)\eta_1),\alpha(z)\rangle
    \\&=\langle ad^\star(T_r(\eta_1))\eta_2,\alpha(T_r(\eta_3))\rangle+(-1)^{|\eta_1||\eta_2|}\langle ad^\star(T_r(\eta_2)\eta_1),\alpha(T_r(\eta_3))\rangle
    \\&=-(-1)^{|\eta_1||\eta_2|}\langle (\alpha^{-1})^{*}(\eta_2), ad(T_r(\eta_1))T_r(\eta_3)\rangle +\langle (\alpha^{-1})^{*}(\eta_1), ad(T_r(\eta_2))T_r(\eta_3)\rangle\\&=(-1)^{|\eta_1||\eta_2|+|\eta_1||\eta_3|}\langle (\alpha^{-1})^{*}(\eta_2), ad(T_r(\eta_3))T_r(\eta_1)\rangle +\langle (\alpha^{-1})^{*}(\eta_1), ad(T_r(\eta_2))T_r(\eta_3)\rangle\\&=(-1)^{|x|(|y|+|z|)}\langle (T_r^{-1}(\alpha(y)), z\cdot x\rangle +\langle T_r^{-1}(\alpha(x)), y\cdot z\rangle\\&=(-1)^{|x|(|y|+|z|)}\mathfrak{B} (\alpha(y), z\cdot x) +\mathfrak{B} (\alpha(x), y\cdot z).
\end{align*}
\end{proof}

\section{$\mathcal{O}$-operators of Hom-Jordan superalgebras and Hom-pre-Jordan superalgebras}\label{Sec4}
In this section, we introduce the notion of Hom-pre-Jordan superalgebras. Then we study the
relations among Hom-Jordan superalgebras, Hom-pre-Jordan superalgebras, and Hom-dendriform superalgebras.
\begin{defn}
\label{hom prejordan}
A Hom-pre-Jordan superalgebra is a Hom-superalgebra  $(\mathcal{A}, \circ,\alpha)$  satisfying, for any $x,y,z,u \in \mathcal{H}(\mathcal{A})$, the following identities
\begin{align}\label{hom prejordan1}
& [\alpha(x) \cdot \alpha(y)]\circ[\alpha(z)\circ \alpha(u)]+(-1)^{|x|(|y|+|z|)} [\alpha(y) \cdot \alpha(z)]\circ[\alpha(x)\circ \alpha(u)]+(-1)^{|z|(|x|+|y|)}[\alpha(z) \cdot \alpha(x)]\circ[\alpha(y)\circ \alpha(u)] \nonumber \\
&= \alpha^2(x)\circ[(y \cdot z) \circ\alpha(u)]+(-1)^{|x|(|y|+|z|)} \alpha^2(y)\circ[(z \cdot x) \circ \alpha(u)]+(-1)^{|z|(|x|+|y|)}\alpha^2(z)\circ[(x \cdot y) \circ \alpha(u)],
\end{align}
\begin{align}&    \alpha^2(x)\circ[\alpha(y)\circ(z\circ u)]+(-1)^{|z|(|x|+|y|)+|x||y|}\alpha^2(z)\circ[\alpha(y)\circ(x\circ u)]+(-1)^{|z||y|}[(x\cdot z)\cdot \alpha(y)]\circ\alpha^2(u) \nonumber\\
\label{hom prejordan2}&=\alpha^2(x)\circ[(y \cdot z) \circ \alpha(u)]+(-1)^{|x|(|y|+|z|)} \alpha^2(y)\circ[(z \cdot x) \circ \alpha(u)]+(-1)^{|z|(|x|+|y|)}\alpha^2(z)\circ[(x \cdot y) \circ \alpha(u)],
\end{align}
where $x \cdot y= x \circ y+(-1)^{|x||y|} y \circ x$. 
In fact, \eqref{hom prejordan1} and \eqref{hom prejordan2} are equivalent to the following equations for any $x,y,z,u\in \mathcal{H}(\mathcal{A})$ respectively
\begin{align}\label{EqvHomPreJordan1}
 & (x,y,z,u)_\alpha^1+(-1)^{|x|(|y|+|z|)}(y,z,x,u)_\alpha^{1}+(-1)^{|z|(|x|+|y|)}(z,x,y,u)_\alpha^1   +(-1)^{|x||y|}(y,x,z,u)_\alpha^{1}\nonumber\\&+(-1)^{|y||z|}(x,z,y,u)_\alpha^{1}+(-1)^{|x||y|+|x||z|+|y||z|}(z,y,x,u)_\alpha^{1}=0,\\
&as_\alpha(\alpha(x),\alpha(y),z\circ u)-(-1)^{|y||z|}as_\alpha(x\circ z,\alpha(y),\alpha(u))+(-1)^{|x|(|y|+|z|)}(y, z,x,u)_\alpha^2 +(-1)^{|x||y|}(y,x,z,u)_\alpha^{2}\nonumber\\&\label{EqvHomPreJordan2}+(-1)^{|x||y|+|x||z|+|y||z|}as_\alpha(\alpha(z),\alpha(y),x\circ u)-(-1)^{|x||z|+|y||z|}as_\alpha(z\circ x,\alpha(y),\alpha(u))=0,
\end{align}where
 \begin{equation*}
  (x,y,z,u)_\alpha^{1}=[\alpha(x)\circ \alpha(y)]\circ[\alpha(z)\circ\alpha(u)]-\alpha^{2}(x)\circ[(y \circ z) \circ \alpha(u)],
\end{equation*} \begin{equation*}
  (x, y,z,u)_\alpha^{2}  =[\alpha(x) \circ \alpha(y)]\circ[\alpha(z)\circ  \alpha(u)]-[\alpha(x)\circ(y\circ z)]\circ \alpha^{2}(u).
\end{equation*}
\end{defn}
\begin{defn}\label{def:HomDendr}
A Hom-dendriform superalgebra is a quadruple $(\mathcal{A}, \prec,\succ, \alpha)$ consisting of $\mathbb{Z}_2$-graded vector space $\mathcal{A}$ on which the operations $\prec, \succ : \mathcal{A}\otimes \mathcal{A} \rightarrow \mathcal{A}$
and $\alpha: \mathcal{A} \rightarrow \mathcal{A}$
 are even linear maps
 satisfying the following equations for  $x, y, z$ in $\mathcal{A}$
\begin{eqnarray}
\label{HomDendriCondition3} \alpha(x)\succ(y\succ z)&=&(x\star y)\succ \alpha (z).\\
\label{HomDendriCondition1}  (x\prec y)\prec \alpha (z)&=& \alpha(x)\prec(y\star z),
 \\ \label{HomDendriCondition2} (x\succ y)\prec \alpha (z)&=&\alpha(x)\succ(y\prec z)
\end{eqnarray}
where
$x\star y=x\prec y+x\succ y$.
\end{defn}
Analogous to the connection between Hom-associative superalgebras and Hom-Jordan
superalgebras, Hom-dendriform superalgebras are closely related to Hom-pre-Jordan
superalgebras.
\begin{prop}\label{HomDendToPreJ}
 Let  $(\mathcal{A}, \prec,\succ, \alpha)$ be a Hom-dendriform superalgebra. Then the product
  \begin{equation*}
 x\circ y= x\succ y +(-1)^{|x||y|}y\prec x, \quad \forall x,y\in \mathcal{H}(\mathcal{A}),
  \end{equation*}
 defines a Hom-pre-Jordan superalgebra structure on $\mathcal{A}$.
\end{prop}
\begin{proof}
Let $x,y,z, u\in \mathcal{H}(\mathcal{A})$, set $$x\star y=x \prec y+ x \succ y$$ and $$x\cdot y=x\circ y+(-1)^{|x||y|} y \circ x=x \star y+(-1)^{|x||y|} y \star x .$$ We have:
\begin{align*}
&[\alpha(x) \cdot \alpha(y)]\circ[\alpha(z)\circ \alpha(u)]+(-1)^{|x|(|y|+|z|)} [\alpha(y) \cdot \alpha(z)]\circ[\alpha(x)\circ \alpha(u)]+(-1)^{|z|(|x|+|y|)}[\alpha(z) \cdot \alpha(x)]\circ[\alpha(y)\circ \alpha(u)] \\&=[\alpha(x) \cdot \alpha(y)]\succ[\alpha(z)\succ \alpha(u)]+ (-1)^{|x|(|y|+|z|)}[\alpha(y) \cdot \alpha(z)]\succ[\alpha(x)\succ \alpha(u)]\\&+ (-1)^{|z|(|x|+|y|)}[\alpha(z) \cdot \alpha(x)]\succ[\alpha(y)\succ \alpha(u)]+[\alpha(x) \cdot \alpha(y)]\succ(-1)^{|u||z|}[\alpha(u)\prec \alpha(z)]\\&+ (-1)^{|u|(|y|+|z|)}[\alpha(x)\circ \alpha(u)]\prec[\alpha(y)\star\alpha(z)]+ (-1)^{|z|(|x|+|u|)+|x|(|y|+|u|)}[\alpha(y) \circ\alpha(u)]\prec(-1)^{|z||x|}[\alpha(x)\star \alpha(z)]\\&+(-1)^{|x|(|y|+|z|)}[\alpha(y) \cdot \alpha(z)]\succ(-1)^{|x||u|}[\alpha(u)\prec \alpha(x)]+ (-1)^{(|x|+|y|)(|z|+|u|)}([\alpha(z) \circ \alpha(u)]\prec(-1)^{|x||y|}[\alpha(y)\star \alpha(x)])\\&+(-1)^{|z|(|x|+|u|)+|x|(|y|+|u|)}[\alpha(y)\circ\alpha(u)]\prec[\alpha(z)\star \alpha(x)]
+(-1)^{|z|(|x|+|y|)}[\alpha(z) \cdot \alpha(x)]\succ(-1)^{|y||u|}[\alpha(u)\prec\alpha(y)]\\&+(-1)^{|u|(|y|+|z|)}[\alpha(x) \circ \alpha(u)]\prec(-1)^{|y||z|}[\alpha(z)\star \alpha(y)]+ (-1)^{(|x|+|y|)(|z|+|u|)}[\alpha(z) \circ \alpha(u)]\prec[\alpha(x)\star \alpha(y)]\\
&=\alpha^2(x)\succ[(y \cdot z) \succ \alpha(u)]+(-1)^{|x|(|y|+|z|)} \alpha^2(y)\succ[(z \cdot x) \succ\alpha(u)]+(-1)^{|z|(|x|+|y|)}\alpha^2(z)\succ[(x \cdot y) \succ\alpha(u)]\\&+\alpha^2(x)\succ(-1)^{|u|(|y|+|z|)}[\alpha(u)\prec(y \star z) ]+ (-1)^{|u|(|z|+|x|)}\alpha^2(y)\succ(-1)^{(|x|+|z|)|u|}[ \alpha(u)\prec(-1)^{|z||x|}(x \star z)]\\&+(-1)^{|u||z|}[(x \cdot y) \circ \alpha(u)]\prec\alpha^2(z)
+(-1)^{|z|(|x|+|y|)}\alpha^2(z)\succ(-1)^{|u|(|y|+|x|)}[\alpha(u)\prec(-1)^{|x||y|}(y \star x) ]\\&+(-1)^{|x|(|y|+|z|)} \alpha^2(y)\succ(-1)^{(|x|+|z|)|u|}[\alpha(u)\prec(z \star x) ]+(-1)^{|x|(|y|+|z|+|u|)}[(y \cdot z) \circ \alpha(u)]\prec\alpha^2(x)\\&+(-1)^{|z|(|x|+|y|)}\alpha^2(z)\succ(-1)^{(|y|+|x|)|u|}[\alpha(u)\prec(x \star y) ]+ \alpha^2(x)\succ(-1)^{|u|(|y|+|z|)}[\alpha(u)\prec(-1)^{|y||z|}(z \star y) ]\\&+(-1)^{|y|(|z|+|u|)+|x||z|}[(z \cdot x) \circ \alpha(u)]\prec\alpha^2(y)\\
&=\alpha^2(x)\circ[(y \cdot z) \circ \alpha(u)]+(-1)^{|x|(|y|+|z|)} \alpha^2(y)\circ[(z \cdot x) \circ \alpha(u)]+(-1)^{|z|(|x|+|y|)}\alpha^2(z)\circ[(x \cdot y) \circ \alpha(u)].
\end{align*}
Then condition \eqref{hom prejordan1} is satisfied. Similarly, we can get \eqref{hom prejordan2}.
\end{proof}
\begin{re}\label{HomPreAlt}
In \cite{SamiOthmenSihemSergei}, the authors define a   Hom-pre-alternative superalgebra $(\mathcal{A}, \prec , \succ ,\alpha  )$  as a generalization of Hom-dendriform superalgebra such that the following identities holds
\begin{align}\label{Hom-preAlt1}
  &\mathfrak{ass}_m(x,y,z)  +(-1)^{|x||y|}\mathfrak{ass}_r(y,x,z)=0,\\
  \label{Hom-preAlt2}
  &\mathfrak{ass}_m(x,y,z)  +(-1)^{|z||y|}\mathfrak{ass}_l(x,z,y)=0,\\
  \label{Hom-preAlt3}
  &\mathfrak{ass}_l(x,y,z)  +(-1)^{|x||y|}\mathfrak{ass}_l(y,x,z)=0,\\
  \label{Hom-preAlt4}
  &\mathfrak{ass}_r(x,y,z)  +(-1)^{|z||y|}\mathfrak{ass}_r(x,z,y)=0,
\end{align}
where for all $x,y,z\in \mathcal{H}(\mathcal{A})$, and 
$x\bullet y=x\succ y+x\prec y$,
\begin{align*}
&\mathfrak{ass}_l(x,y,z)=(x\bullet y) \succ \alpha(z) -\alpha(x)\succ(y \succ z), \\&\mathfrak{ass}_m(x,y,z)=(x \succ y) \prec \alpha(z) - \alpha(x) \succ (y \prec z),\\
&\mathfrak{ass}_r(x,y,z)=
(x\prec y)\prec \alpha(z) - \alpha(x)\prec(y \bullet z).
\end{align*}
The above construction is true in Hom-pre-alternative superalgebra. 
\end{re}

\begin{prop}\label{Hom-pre-Jor-to-Hom-Jor}
 Let $(\mathcal{A},\circ,\alpha)$ be a Hom-pre-Jordan superalgebra. Then the
product
 \begin{equation}\label{compCondition}
 x\cdot y= x\circ y +(-1)^{|x||y|} y\circ x, \quad \forall x,y\in \mathcal{H}(\mathcal{A}),
  \end{equation}
  defines a Hom-Jordan superalgebra on $\mathcal{A}$, which is called the associated Hom-Jordan
superalgebra of $(\mathcal{A}, \circ,\alpha)$ and $(\mathcal{A}, \circ,\alpha)$ is called a compatible Hom-pre-Jordan superalgebra structure of the Hom-Jordan superalgebra $(\mathcal{A},\cdot,\alpha)$.
\end{prop}
\begin{proof}
 Let $x,y,z,u\in \mathcal{H}(A)$. We can prove that
 \begin{align*}
& (-1)^{|z|(|x|+|u|)}((x \cdot y) \cdot \alpha(u)) \cdot \alpha^2(z)  + (-1)^{|x|(|y|+|u|)}((y \cdot z) \cdot \alpha(u)) \cdot \alpha^2(x) +(-1)^{|y|(|z|+|u|)}((z \cdot x) \cdot \alpha(u))\cdot\alpha^2(y) \nonumber\\
&= (-1)^{|z|(|x|+|u|)}[\alpha(x) \cdot \alpha(y)]\cdot [\alpha(u) \cdot \alpha(z)] +(-1)^{|x|(|y|+|u|)} [\alpha(y) \cdot \alpha(z)]\cdot [\alpha(u) \cdot \alpha(x)] \\&+(-1)^{|y|(|z|+|u|)} [\alpha(z )\cdot\alpha( x)]\cdot [\alpha(u) \cdot\alpha(y)]=0
 \end{align*}
 if and only if $l_1+l_2+l_3=r_1+r_2+r_3$ where
 \begin{align*}
  l_1 & =(-1)^{|z|(|x|+|u|)}[(x\cdot y)\cdot \alpha(u)]\circ\alpha^2(z)+(-1)^{|z||x|+|u|(|y|+|z|)} \alpha^2(x)\circ[\alpha(u)\circ(y\circ z)]\\&+(-1)^{|x|(|y|+|z|)+|u|(|z|+|x|)}\alpha^2(y)\circ[\alpha(u)\circ(x\circ z)] ,\\
  l_2 & =(-1)^{|y|(|z|+|u|)}[(z\cdot x)\cdot \alpha(u)]\circ\alpha^2(y)+(-1)^{|z|(|x|+|y|)+|u|(|y|+|z|)}\alpha^2(x)\circ[\alpha(u)\circ(z\circ y)]\\&+(-1)^{|z||y|+(|u|(|x|+|y|)}\alpha^2(z)\circ[\alpha(u)\circ(x\circ y)] ,\\
  l_3 & =(-1)^{(|x|(|y|+|u|)}[(y\cdot z)\cdot \alpha(u)]\circ\alpha^2(x)+(-1)^{|x||y|+|u|(|z|+|x|)} \alpha^2(y)\circ[\alpha(u)\circ(z\circ x)]\\&+(-1)^{|y|(|z|+|x|)+|u|(|x|+|y|)}\alpha^2(z)\circ[\alpha(u)\circ(y\circ x)] ,
\end{align*}
and
\begin{align*}
   r_1&=\sum_{x,y,u}(-1)^{|z|(|x|+|u|)}(\alpha(x)\cdot\alpha(y))\circ(\alpha(u)\circ\alpha(z)),  \\
   r_2&=\sum_{x,z,u}(-1)^{|y|(|z|+|u|)}(\alpha(z)\cdot\alpha(x))\circ(\alpha(u)\circ\alpha(y)),\\
   r_3&=\sum_{y,z,u}(-1)^{|x|(|y|+|u|)}(\alpha(y)\cdot\alpha(z))\circ(\alpha(u)\circ\alpha(x)).
\end{align*}
Note that
{\small\begin{eqnarray*}
 l_1&=&(-1)^{|z|(|x|+|u|)}[(x\cdot y)\cdot \alpha(u)]\circ\alpha^2(z)+(-1)^{|z||x|+|u|(|y|+|z|)} \alpha^2(x)\circ[\alpha(u)\circ(y\circ z)]\\&&+(-1)^{|x|(|y|+|z|)+|u|(|z|+|x|)}\alpha^2(y)\circ[\alpha(u)\circ(x\circ z)]\\&=&(-1)^{|z|(|x|+|u|)}[(x\circ y+(-1)^{|x||y|}y\circ x)\cdot \alpha(u)]\circ\alpha^2(z)+(-1)^{|z||x|+|u|(|y|+|z|)} \alpha^2(x)\circ[\alpha(u)\circ(y\circ z)]\\&&+(-1)^{|x|(|y|+|z|)+|u|(|z|+|x|)}\alpha^2(y)\circ[\alpha(u)\circ(x\circ z)]\\&=&(-1)^{|z|(|x|+|u|)}[(x\circ y+(-1)^{|x||y|}y\circ x)\circ \alpha(u)+(-1)^{|u|(|x|+|y|)}\alpha(u)\circ (x\circ y+(-1)^{|x||y|}y\circ x)]\circ\alpha^2(z)\\&&+(-1)^{|z||x|+|u|(|y|+|z|)} \alpha^2(x)\circ[\alpha(u)\circ(y\circ z)]+(-1)^{|x|(|y|+|z|)+|u|(|z|+|x|)}\alpha^2(y)\circ[\alpha(u)\circ(x\circ z)]\\&=&(-1)^{|z|(|x|+|u|)}[(x\circ y)\circ \alpha(u)]\circ\alpha^2(z)+(-1)^{|z|(|x|+|u|)+|x||y|}[(y\circ x)\circ \alpha(u)]\circ\alpha^2(z)\\&&+(-1)^{|z|(|x|+|u|)+|u|(|x|+|y|)}[\alpha(u)\circ (x\circ y)]\circ \alpha^2(z)+(-1)^{|z|(|x|+|u|)+|u|(|x|+|y|)+|x||y|}[\alpha(u)\circ (y\circ x)]\circ\alpha^2(z)\\&&+(-1)^{|z||x|+|u|(|y|+|z|)} \alpha^2(x)\circ[\alpha(u)\circ(y\circ z)]+(-1)^{|x|(|y|+|z|)+|u|(|z|+|x|)}\alpha^2(y)\circ[\alpha(u)\circ(x\circ z)]\\&\stackrel{\eqref{hom prejordan1}}{=}&(-1)^{|z|(|x|+|u|)}(\alpha(x)\circ \alpha(y))\circ(\alpha(u)\circ\alpha(z))+(-1)^{|z|(|x|+|u|)+|x||y|}(\alpha(y)\circ \alpha(x))\circ (\alpha(u)\circ\alpha(z))\\&&+(-1)^{|z|(|x|+|u|)+|u|(|x|+|y|)}(\alpha(u)\circ \alpha(x))\circ(\alpha(y)\circ \alpha(z))+(-1)^{|z|(|x|+|u|)+|u|(|x|+|y|)+|x||y|}(\alpha(u)\circ \alpha(y))\circ(\alpha(x)\circ\alpha(z))\\&&+(-1)^{|z||x|+|u|(|y|+|z|)} (\alpha(x)\circ\alpha(u))\circ(\alpha(y)\circ\alpha(z))+(-1)^{|x|(|y|+|z|)+|u|(|z|+|x|)}(\alpha(y)\circ\alpha(u))\circ(\alpha(x)\circ \alpha(z))\\&=&(-1)^{|z|(|x|+|u|)}(\alpha(x)\cdot \alpha(y))\circ (\alpha(u)\circ\alpha(z))+(-1)^{|z||x|+|u|(|y|+|z|)} (\alpha(x)\cdot\alpha(u))\circ(\alpha(y)\circ\alpha(z))\\&&+(-1)^{|x|(|y|+|z|)+|u|(|z|+|x|)}(\alpha(y)\cdot\alpha(u))\circ(\alpha(x)\circ \alpha(z)))\\&=&\sum_{x,y,u}(-1)^{|z|(|x|+|u|)}(\alpha(x)\cdot\alpha(y))\circ(\alpha(u)\circ\alpha(z))=r_1.
\end{eqnarray*}}
Similarly, we have  $l_2=r_2$, $l_3=r_3$.
\end{proof}
By the definition of a Hom-Jordan superalgebras and $\mathcal{J}$-module, we immediately obtain the following result.
\begin{prop}\label{rep pre}
A triple $(\mathcal{J},\circ, \alpha )$ is a Hom-pre-Jordan superalgebra if and only
if $(\mathcal{J},\cdot, \alpha )$ where, $"\cdot"$ is define by \eqref{compCondition} is a Hom-Jordan superalgebra and $(\mathcal{J},L , \alpha )$ is a $\mathcal{J}$-module of $(\mathcal{J},\cdot, \alpha )$, where $L$ denotes the left multiplication operator on $\mathcal{J}$ define by $$L_{x}y=x\circ y,\quad \forall x,y\in \mathcal{H}(\mathcal{J}).$$
\end{prop}
\begin{proof}
By Proposition \ref{Hom-pre-Jor-to-Hom-Jor}, $(\mathcal{J},\cdot, \alpha )$ is a Hom-Jordan superalgebra.\\
For any $x\in \mathcal{H}(\mathcal{J})$, we have $\alpha L_x(y)=\alpha(x\circ y)=\alpha(x)\circ\alpha(y)=L_{\alpha(x)}\alpha(y)$, which gives the identity \eqref{Hom-jordan-representation3}.\\
Let $x,y,z,t\in \mathcal{H}(\mathcal{J})$, we have
\begin{align*}
  &((-1)^{|x||z|} L_{\alpha(x)\circ \alpha(y)}L_{\alpha(z)}\alpha+(-1)^{|x||y|}L_{\alpha(y)\circ \alpha(z)}L_{\alpha(x)}\alpha + (-1)^{|z||y|} L_{\alpha(z)\circ \alpha(x)}L_{(\alpha(y)}\alpha)(t)\\=& (-1)^{|x||z|}  (\alpha(x)\circ \alpha(y))\circ\alpha(z,t)+(-1)^{|x||y|}(\alpha(y)\circ \alpha(z))\circ(\alpha(x\circ t)+ (-1)^{|z||y|} L_{\alpha(z)\circ \alpha(x)}L_{(\alpha(y)}\alpha)(t)\\=&(-1)^{|x||z|}\alpha^{2}(x)\circ((y\circ z)\circ \alpha(t))+(-1)^{|x||y|}\alpha^{2}(y)\circ((z\circ x)\circ \alpha(t))+(-1)^{|z||y|}\alpha^{2}(z)\circ((x\circ y)\circ \alpha(t))\\=&(-1)^{|x||z|} L_{\alpha^{2}(x)}L_{(y\circ z)} \alpha+(-1)^{|x||y|}L_{\alpha^{2}(y)}L_{(z\circ x)}\alpha + (-1)^{|z||y|} L_{\alpha^{2}(z)}L_{(x\circ y)}\alpha)(t).
\end{align*}
Then, the identity \eqref{Hom-jordan-representation2} is satisfied. Similarly we can show that the identity \eqref{Hom-jordan-representation3} is satisfied. 
Therefore $(\mathcal{J},L , \alpha )$ is a $\mathcal{J}$-module of $(\mathcal{J},\cdot, \alpha )$.
\end{proof}
Furthermore, we can construct Hom-pre-Jordan superalgebras from  $\mathcal{O}$-operators
of Hom-Jordan superalgebras.
\begin{prop}\label{induced hompre}
 Let $(\mathcal{J},\cdot,\alpha)$ be a Hom-Jordan superalgebra and $(V,\pi, \beta )$ be a $\mathcal{J}$-module.
Let $T: V\to \mathcal{J}$ be an $\mathcal{O}$-operator of $\mathcal{J}$ associated to $(V,\pi, \beta )$. Then $(V,\star,\beta)$ is a
Hom-pre-Jordan superalgebra, where
\begin{align}
  u\star v=\pi(T(u))v, \quad \forall u,v\in \mathcal{H}(V).
\end{align}
Therefore, there exists an associated Hom-Jordan superaglebra structure on $V$ given by Eq.\eqref{compCondition}
and $T$ is a homomorphism of Hom-Jordan superalgebras. Furthermore, there is an
induced Hom-pre-Jordan superalgebra structure on $T(V)$ given by
\begin{align}
   T(u)\cdot T(v)=T(u \star v), \,\,\ \forall u,v\in\mathcal{H}(V).
\end{align}
\end{prop}
\begin{proof}
Let $u,v,w,a\in \mathcal{H}(V)$ and put $x=T(u)$, $y=T(v)$, $z=T(w)$ and $u\bullet v=u\star v+(-1)^{|u||v|}v\star u$.\\
Hence we have
\begin{align*}
  (\beta(u)\bullet \beta(v))\star(\beta(w)\star\beta(a))
  &=(\pi(T(\beta(u)\beta(v)+(-1)^{|u||v|}\pi(T(\beta(v)))\beta(u)))\star(\pi(T(\beta(w)))\beta(a))\\
&=\pi(T(\pi(T(\beta(u))\beta(v)+(-1)^{|u||v|}\pi(T(\beta(v)))\beta(u))))\pi(T(\beta(w)))\beta(a)  \\
   &=\pi(T(\beta(u))\circ T(\beta(v)))\pi(T(\beta(w)))\beta(a)\\
   &=\pi(\alpha(x)\circ \alpha(y))\pi(\alpha(z))\beta(a),
\end{align*}
\begin{align*}
\beta^2(u)\star[(v\bullet w)\star \beta(a)]
&=\pi(T(\beta^2(u)))\pi(T(\pi(T(v))w +(-1)^{|w||v|}\pi(T(w))v )) \beta(a) \\
&=\pi(T(\beta^2(u)))\pi(T(v)\circ T(w)) \beta(a)\\
&=\pi(\alpha^2(x))\pi(y\circ z) \beta(a)
\end{align*}
and
\begin{align*}
  \beta^2(u)\star[\beta(v)\star(w\star a)] & =\pi(T(\beta^2(u)))\pi(T(\beta(v)))\pi(T(w)) a \\
    & =\pi(\alpha^2(x))\pi(\alpha(y))\pi(z) a.
\end{align*}
Then
\begin{align*}
   & (\beta(u)\bullet \beta(v))\star(\beta(w)\star\beta(a)) + (-1)^{(|v|+|w|)|u|}(\beta(v)\bullet \beta(w))\star(\beta(u)\star\beta(a))+(-1)^{|w|(|u|+|v|)}(\beta(w)\bullet \beta(u))\star(\beta(v)\star\beta(a))  \\
   & =(-1)^{|z|(|x|+|y|)}\pi(\alpha(x)\circ \alpha(y))\pi(\alpha(z))\beta(a)+(-1)^{|x|(|y|+|z|)}\pi(\alpha(y)\circ \alpha(z))\pi(\alpha(x))\beta(a)\\&+(-1)^{|z|(|x|+|y|)}\pi(\alpha(z)\circ \alpha(x))\pi(\alpha(y))\beta(a) \\
   & =\pi(\alpha^2(x))\pi(y\circ z) \beta(a)+(-1)^{|x|(|z|+|y|)}\pi(\alpha^2(y))\pi(z\circ x) \beta(a)+(-1)^{|z|(|x|+|y|)}\pi(\alpha^2(z))\pi(x\circ y) \beta(a)\\
   &=\beta^2(u)\star[(v\bullet w)\star \beta(a)]+(-1)^{|u|(|v|+|w|)}\beta^2(v)\star[(w\bullet u)\star\beta(a)]+(-1)^{|w|(|u|+|v|)}\beta^2(w)\star[(u\bullet v)\star \beta(a)],
\end{align*}
and
\begin{align*}
   &\beta^2(u)\star[\beta(v)\star(w\star a)]+(-1)^{|u|(|w|+|v|)+|w||v|}\beta^2(w)\star[\beta(v)\star(u\star a)]+ (-1)^{|w||v|}[(u\bullet w)\bullet \beta(v)]\star \beta^2(a) \\
   & =\pi(\alpha^2(x))\pi(\alpha(y))\pi(z) a+(-1)^{|x|(|y|+|z|)+|y||z|}\pi(\alpha^2(z))\pi(\alpha(y))\pi(x) a+(-1)^{|z||y|}\pi([(x\circ z)\circ \alpha(y)]) \beta^2(a) \\
   & =\pi(\alpha^2(x))\pi(y\circ z) \beta(a)+(-1)^{|z|(|x|+|y|)}\pi(\alpha^2(z))\pi(x\circ y) \beta(a)+(-1)^{|x|(|y|+|z|)}\pi(\alpha^2(y))\pi(z\circ x) \beta(a)\\
   &=\beta^2(u)\star[(v\bullet w)\star\beta(a)]+(-1)^{|x|(|y|+|z|)}\beta^2(v)\star[(w\bullet u)\star\beta(a)]+(-1)^{|z|(|x|+|y|)}\beta^2(w)\star[(u\bullet v)\star \beta(a)].
\end{align*}
Thus, $(V,\star,\beta)$ is a Hom-pre-Jordan superalgebra.The other conclusions follow immediately.
\end{proof}
\begin{cor}\label{RBHomJor}
Let $(\mathcal{J},\cdot,\alpha )$ be a Hom-Jordan superalgebra, and $R$ be a  Rota-Baxter operator of weight z\'{e}ros on $\mathcal{J}$. Then there is a Hom-pre-Jordan superalgebra structure given by
\begin{align*}
x \circ y = R(x) \cdot y, \quad \forall  x, y \in \mathcal{H}(\mathcal{J}).
\end{align*}
\end{cor}
\begin{ex}\label{pre-Jor-from-RB}
Let us consider the Hom-Jordan superalgebra $(K_3,\cdot_\alpha,\alpha)$ given in Example \ref{ex-hom-jor}. Define the even linear map $R:K_3\to K_3$ with respect to the basis $\{e,x,y\}$ by
$$R(e)=2e,~~~~R(x)=\lambda y,~~~~R(y)=0,$$
where $\lambda\in\mathbb{K}$. Then, the even bilinear map $\circ:K_3\to K_3$ defined by 
$$e\circ e=R(e)\cdot_\alpha e=2e,\;\;e\circ x=R(e)\cdot_\alpha x=\frac{1}{c}x,\;\;e\circ y=R(e)\cdot_\alpha y=\frac{1}{c}y,$$
$$x\circ e=R(x)\cdot_{\alpha }e=\frac{\lambda}{2c}y,\;\;x\circ x=R(x)\cdot_{\alpha }x=\frac{-\lambda}{c^2}e,$$
defines a Hom-pre-Jordan superalgebra structure on $K_3$.
\end{ex}
Next we will give a sufficient and necessary conditions for the existence of a compatible Hom-pre-Jordan superalgebra structure on a Hom-Jordan superalgebra.
\begin{prop}\label{compatible}
Let $(\mathcal{J},\cdot,\alpha)$ be a Hom-Jordan superalgebra and $(V,\pi, \beta )$ be a $\mathcal{J}$-module. Then there exists a compatible Hom-pre-Jordan superalgebra structure on $\mathcal{J}$ if and only if, there exists an invertible
$\mathcal{O}$-operator $T:V\rightarrow \mathcal{J}$ associated to $(V,\pi, \beta )$. Furthermore, the compatible Hom-pre-Jordan superalgebra structure on $\mathcal{J}$ is given by
\begin{align}
  x\circ y=T(\pi(x))T^{-1}(y), \quad \forall x,y\in \mathcal{H}(\mathcal{J}).
\end{align}
\end{prop}
\begin{proof}
This is a direct computation, we apply Proposition \ref{induced hompre} for $T(V)=\mathcal{J}$.
Conversely, the identity map is an $\mathcal{O}$-operator of $\mathcal{J}$.
\end{proof}

The following conclusion reveals the relationship between Hom-pre-Jordan superalgebras and the Hom-Jordan superalgebras.
\begin{prop}
Let $(\mathcal{J},\cdot,\alpha)$ be a Hom-Jordan superalgebra, and $\mathfrak{B}$ a nondegenerate super skew-symmetric bilinear form satisfying Eq.\eqref{2-cocycle}. Then, there exists a compatible Hom-pre-Jordan superalgebra structure
$"\circ"$ on $\mathcal{J}$ given by
\begin{align}
  \mathfrak{B}(x\circ y, \alpha(z))=-(-1)^{|x||y|}\mathfrak{B}(\alpha(y),x\cdot  z)\quad \forall x, y, z\in  \mathcal{H}(\mathcal{J}).
\end{align}
\end{prop}
\begin{proof}
   By using Eq.\eqref{quadratic} and Eq.\eqref{2-cocycle}, we obtain that $T$ is an invertible $\mathcal{O}$-operator associated to the coadjoint representation $(\mathcal{J^{*}},ad^\star,(\alpha^{-1})^{*})$. By Proposition \ref{compatible}, there exists a compatible Hom-pre-Jordan superalgebra structure on $\mathcal{J}$ given by $x\circ y=T(ad^\star(x)T^{-1}(y))$ for any $x,y,z\in \mathcal{H}(\mathcal{J})$. Also, there exist $\eta_1,\eta_2,\eta_3\in \mathcal{H}(\mathcal{J^{*}})$ such that $x=T(\eta_1),y=T(\eta_2),z=T(\eta_3)$ and we have
\begin{align*}
    \mathfrak{B}(x\circ y, \alpha(z))&=\mathfrak{B}(T(ad^\star(x)T^{-1}(y)),\alpha(z))\\&=\mathfrak{B}(T(ad^*(\alpha(x))(\alpha^{-2})^*T^{-1}(y)),\alpha(z))
    &
    \\&=\langle ad^*(\alpha T(\eta_1))(\alpha^{-2})^*\eta_2,\alpha(z)\rangle\\&=\langle ad^*(\alpha T(\eta_1))(\alpha^{-2})^*\eta_2,\alpha(T(\eta_3))\rangle\\&=\langle (\alpha^{-2})^*\eta_2,ad(\alpha T(\eta_1))(\alpha(T(\eta_3))) \rangle\\&=\langle (\alpha^{-2})^*T^{-1}(y),\alpha T(\eta_1)\cdot\alpha T(\eta_3) \rangle\\&=\langle (\alpha^{-1})^*T^{-1}(y), T(\eta_1)\cdot T(\eta_3) \rangle\\&=\langle T^{-1}(\alpha(y)), T(\eta_1)\cdot T(\eta_3) \rangle\\&=\mathfrak{B}(\alpha(y), x\cdot z) \rangle.
\end{align*}
This completes the proof.
\end{proof}
Summarizing the above study in this section, we have the following commutative diagram:

\begin{equation*}\label{diagramhommalcev}
    \begin{split}
\resizebox{14cm}{!}{\xymatrix{
\ar[rr] \mbox{\bf Hom-Dend superalg }\ar[d]_{\mbox{}}\ar[rr]^{\mbox{ \quad Prop \ref{HomDendToPreJ}\quad\quad}}_{\mbox{ \quad}}
                && \mbox{\bf  Hom-pre-Jor superalg  }\ar[d]_{\mbox{Prop \ref{Hom-pre-Jor-to-Hom-Jor}}}\\
\ar[rr] \mbox{\bf Hom-Ass superalg }\ar@<-1ex>[u]_{\mbox{  }}\ar[rr]^{\mbox{\quad\quad\quad\quad {\small SuperAntiCom}\quad\quad \quad \quad }}
                && \mbox{\bf Hom-Jor superalg  }\ar@<-1ex>[u]_{\mbox{Cor \ref{RBHomJor} }}}
}
 \end{split}\end{equation*}
\begin{cor}
 Let $(\mathcal A, \cdot, \alpha)$ be a Hom-alternative superalgebra and $R : \mathcal A \to \mathcal A$ be a Rota-Baxter
operator of weight $0$. Define the  multiplications $\succ$ and $\prec$ on $\mathcal A$ by $x \prec y = x \cdot R(y)$ and $x \succ y = R(x) \cdot y$, for all
$x,y\in \mathcal A$. Then $(\mathcal A, \succ, \prec, \alpha)$ is a Hom-pre-alternative
superalgebra.

Moreover, $R$ is Rota-Baxter operator of weight $0$ on the Hom-Jordan admissible superalgebra $(\mathcal A, \star, \alpha)$ and   $(\mathcal A, \circ, \alpha)$ be its  associated Hom-pre-Jordan superalgebra given in    Proposition \ref{Hom-pre-Jor-to-Hom-Jor}, where $$x \circ y = x \succ y + (-1)^{|x||y|}y, \prec x\quad \forall x,y\in \mathcal H(\mathcal A).$$ 
\end{cor}
A Hom-Jordan  superalgebras, Hom-alternative superalgebras, Hom-pre-Jordan superalgebras and
Hom-pre-alternative superalgebras are closely related as follows (in the sense of commutative diagram
of categories):

{
\begin{equation*}\label{diagramhommalcev}
    \begin{split}
\resizebox{14cm}{!}{\xymatrix{
\ar[rr] \mbox{\bf Hom-pre-Alt superalg }\ar[d]_{\mbox{{\small\big(\cite{SamiOthmenSihemSergei}, Thm 4.3\big)}}}\ar[rr]^{\mbox{ \quad Rem \ref{HomPreAlt}\quad\quad}}_{\mbox{ \quad}}
                && \mbox{\bf  Hom-pre-Jor superalg  }\ar[d]_{\mbox{Prop \ref{Hom-pre-Jor-to-Hom-Jor}}}\\
\ar[rr] \mbox{\bf Hom-alt superalg }\ar@<-1ex>[u]_{\mbox{{\small\big(\cite{SamiOthmenSihemSergei}, Prop 4.8\big) }}}\ar[rr]^{\mbox{\quad\quad\quad\quad Prop \ref{FromAltToJorMalcev}\quad\quad\quad\quad  }}
                && \mbox{\bf Hom-Jordan superalg . }\ar@<-1ex>[u]_{\mbox{Cor \ref{RBHomJor} }}}
}
 \end{split}\end{equation*}}


\end{document}